\documentclass[11pt]{amsart}
\usepackage[mathcal]{euscript}
\usepackage{amsthm, color}
\usepackage{amsfonts}
\usepackage{amsmath,amssymb}
\usepackage{indentfirst,latexsym,bm,amsthm,graphicx,colortbl}
\usepackage{fancyhdr}
\usepackage{times}
\usepackage{amsmath,amsfonts,amssymb,amsthm}
\usepackage{epsfig}
\usepackage[all]{xy}
\newtheorem{theorem}{Theorem}[section]
\newtheorem{lemm}[theorem]{Lemma}
\newtheorem{prop}[theorem]{Proposition}
\newtheorem{theo}[theorem]{Theorem}
\theoremstyle{definition}
\newtheorem{defi}[theorem]{Definition}

\theoremstyle{remark}
\newtheorem{remark}[theorem]{Remark}

\numberwithin{equation}{section}

\newcommand{\la}{\lambda}

\newcommand{\vep}{\varepsilon}

\def\ot{\otimes}
\def\ra{\rangle}
\def\la{\langle}

\begin{document}

\title[On Hopf algebraic structures of quantum toroidal algebras]
{On Hopf algebraic structures of quantum toroidal algebras}

%%% ----------------------------------------------------------------------
\author[Jing]{Naihuan Jing}
\address{Department of Mathematics,
   North Carolina State University,
   Ra\-leigh, NC 27695-8205, USA}%\\
\email{jing@math.ncsu.edu}

\author[Zhang]{Honglian Zhang$^\star$}
\address{Department of Mathematics, Shanghai University,
Shanghai 200444, China} \email{hlzhangmath@shu.edu.cn}

\thanks{$^\star$ H.Zhang, Corresponding Author}

%    General info
\date{Version on Jan. 13, 2021}

%%%%%%%%%%%%%%%%%%%%%%%%%%%%%%%%%%%%%%%%%%%%%%%%%%%%%%%%%%%%%%%%%%%%%%%%
%\footnote{Corresponding author.}%
%%%%%%%%%%%%%%%%%%%%%%%%%%%%%%%%%%%%%%%%%%%%%%%%%%%%%%%%%%%%%%%%%%%%%%%%
\begin{abstract}
We define an algebra $\mathcal{U}_0$ using a simplified set of generators for the quantum toroidal algebra $U_q(sl_{n+1}, tor)$ and show that there exists  an epimorphism from $\mathcal{U}_0$ to $U_q(sl_{n+1}, tor)$.
We derive a closed formula of  the comultiplication on the generators of $\mathcal{U}_0$ that
extends that of the quantum affine algebra $U_q(\widehat{sl}_{n+1})$. As a consequence,
we show that $\mathcal{U}_0$ is a Hopf algebra for $n=1, 2$ and give conjectural formulas in the general case.
We further show that $\mathcal{U}_0$ is isomorphic to a double algebra.
\end{abstract}

\maketitle

\section{ Introduction}

The quantum toroidal algebras were first introduced via geometric realization in type A \cite{GKV} and
then through the McKay correspondence in the general $ADE$ types \cite{FJW} in terms of Drinfeld generators
and relations. Their structures
have been studied in various works \cite{VV, H1}, and in particular
the horizontal and vertical quantum affine subalgebras have played an important role in subsequent investigations.
The quantum toroidal algebra is also closely related to
the quantum Kac-Moody algebra \cite{J3}, which has a Hopf algebra structure in terms of
Drinfeld-Jimbo generators. It is natural to expect that the quantum toroidal algebras
are also quantum groups in Drinfeld's sense \cite{D1}, and the quantum group
structure should naturally extend to that of quantum affine algebras.

In \cite{H2, FJM1, FJM2, M1, M3}, tensor products of modules of quantum toroidal algebras are studied in terms of Drinfeld-type
(infinite) comultiplication (cf \cite{DI, J2}).
These works also indicate more tensor products of modules could be constructed if there exists a finite type comultiplication.
However, it has been a long-standing problem to find such a (finite) Hopf algebra structure
for the quantum toroidal algebra. The difficulties might be due to the fact that
the quantum toroidal algebras have complicated high degree Serre relations (see \cite{M2}).

Besides the Drinfeld-Jimbo presentation, Beck \cite{B} has given a general formula for the Hopf algebra structure of the quantum affine algebra
in terms of the universal R-matrix and the braid group action. Through the Ding-Frenkel isomorphism \cite{DF}
the Hopf algebra structure is also given in terms of the
 $L^{\pm}(z)$-generators \cite{FR}.
 In a different approach \cite{JZ3}, the authors
 have computed the  Hopf comultiplication of
 the two-parameter quantum affine algebras explicitly in terms of some simpler Drinfeld generators \cite{D2}. The special case of the
 comultiplication formula gives the same (Drinfeld-Jimbo) Hopf algebra structure for the quantum affine algebras.
The goal of this paper is to generalize and study the Hopf algebra structure for the quantum toroidal algebras of type A in terms of its Drinfeld generators.
%With this Hopf algebra structure, the quantum toroidal algebra can be then realized as a double quantum algebra.

As our comultiplication stems out of the usual Drinfeld-Jimbo comultiplication over the
quantum affine algebra, it is compatible with
the $q$-characters \cite{FR1}. We define an algebra $\mathcal{U}_0$ using a simplified set of generators for the quantum toroidal algebra $U_q(sl_{n+1}, tor)$. We prove that there exists  an epimorphism between $\mathcal{U}_0$ and  quantum toroidal algebra $U_q(sl_{n+1}, tor)$
We have shown that this does give a Hopf algebra structure of $\mathcal{U}_0$ for the
case of $n=1$ and $n=2$.
We also formulate a conjectural formula of the coproduct in the general case, and it is also
interesting to see if this type of comultiplication exists for other quantum toroidal algebras following \cite{FR1, FM}. We remark that
our conjectured formulas seem to be related with rooted trees.

Guay, Nakajima and Wendlandt \cite{GNW} have recently given a Hopf algebra structure on the affine Yangian algebra $Y(\widehat{\mathfrak g})$
in terms of the Drinfeld generators, which generalizes the comultiplication of the current algebra $\widehat{\mathfrak g}[t]$ (see also \cite{KT}).
Our result would induce a corresponding
Hopf algebra structure for the double affine Yangian algebra $DY(\widehat{\mathfrak g})$ in view of \cite{GTL, GM}, which
generalizes the multiplication of the loop algebra $\widehat{\mathfrak g}[t, t^{-1}]$.

The paper is organized as follows. In section 2, we review some basic results for the quantum affine algebra
\cite{JZ3} to prepare for later discussions. In section
4 we define an algebra $\mathcal{U}_0$ using a subset of Drinfeld generators for the quantum toroidal algebra $U_q(sl_{n+1},tor)$
and we prove that there exists  an epimorphism between $\mathcal{U}_0$ and $U_q(sl_{n+1}, tor)$ . In section 5, we define the comutiplication $\Delta$ on the generators of the algebra  $\mathcal{U}_0$ generated by the simpler set of generators, and show that the Hopf algebra structure for $n=1, 2$. In section 6, we prove that $\mathcal{U}_0$ is characterized as a Drinfeld double ${\mathcal D}({\mathcal B}, {\mathcal B'})$ with respect to a skew-dual paring based on the comultiplication.

\section{Quantum affine algebras}

The quantum affine algebras have the Hopf algebra structure defined by Drinfeld and Jimbo
in terms of the Chevalley generators (see also \cite{B}).  %gave a formula of the comultiplication.
Drinfeld gave an infinite coproduct on the completion algebra (cf. \cite{J1}).
Now we recall a commultiplication for the Drinfeld generators \cite{JZ3}
coming from the two-parameter quantum affine algebras.

 Let $A=(a_{ij})_{i, j\in I} $ be the generalized Cartan matrix
of the (untwisted) affine Lie algebra $\hat{\mathfrak{g}}$, where
$I =\{0, 1, \cdots, n\}$ and $I_0=I\backslash\{0\}$ are the index sets of simple roots $\alpha_i$
of $\hat{\mathfrak{g}}$ and the finite dimensional simple Lie algebra $\mathfrak g$ respectively.
The canonical bilinear form $( \ , \ )$ of $\hat{\mathfrak{g}}$
 satisfies $(\alpha_i, \alpha_j)=d_ia_{ij}$ and $
(\delta, \alpha_i)=(\delta, \delta)=0$
%\end{equation*}
where $\delta$ is the canonical imaginary root. So $d_i=\frac12(\alpha_i, \alpha_i)$
and let $q_i=q^{d_i}$.
Let $\mathfrak{h}$ be the Cartan subalgebra of $\hat{\mathfrak{g}}$ and  $\mathfrak{h}^*$ the dual Cartan subalgebra.

\begin{defi} The quantum affine algebra $U_{q}(\hat{\mathfrak{g}})$ is the unital
associative algebra over $\mathbb C(q)$ generated by $e_i,\, f_i,\, K_i^{\pm 1},\, (i\in I)$, $D^{\pm1}$ and the central elements
$\gamma^{\pm\frac{1}2}$ subject to the following relations:
\begin{align*}
 &(\mathcal{X}1) &&
  K_i\,K_i^{-1}=K_i^{-1}\,K_i=1, [\,K_i^{\pm 1},K_j^{\,\pm 1}\,]=[\,K_i^{\pm1},
D^{\pm1}\,]=0,\\
 &(\mathcal{X}2)  && D\,e_i\,D^{-1}=q_i^{\delta_{0i}}\,e_i,\qquad D\,f_i\,D^{-1}=q_i^{-\delta_{0i}}\,f_i,\\
 &                && K_je_iK_j^{-1}=q_i^{a_{ij}}\,e_i,  \qquad K_jf_iK_j^{-1}=q_i^{-a_{ij}}\,f_i.\\
 & (\mathcal{X}3) && [\,e_i, f_j\,]=\frac{\delta_{ij}}{q_i-q_i^{-1}}(K_i-K_i^{-1}). \\
&(\mathcal{X}4)  &&
\bigl(\text{ad}_l\,e_i\bigr)^{1-a_{ij}}\,(e_j)=0,\quad \bigl(\text{ad}_r\,f_i\bigr)^{1-a_{ij}}\,(f_j)=0, \quad i\ne j\in I,
\end{align*}
where the left-adjoint $\text{ad}_l\,e_i$
and the right-adjoint $\text{ad}_r\,f_i$ are defined by
\begin{align*}
\text{ad}_{ l}\,a\,(b)=\sum_{(a)}a_{(1)}\,b\,S(a_{(2)}), \quad
\text{ad}_{ r}\,a\,(b)=\sum_{(a)}S(a_{(1)})\,b\,a_{(2)}, \quad
\forall\; a, b\in U(\hat{\mathfrak g}),
\end{align*}
where $\Delta(a)=\sum_{(a)}a_{(1)}\ot a_{(2)}$ is the comultiplication
of $U_{q}(\hat{\mathfrak{g}})$ given below.
\end{defi}

The Hopf algebra structure of $U_{q}(\hat{\mathfrak{g}})$ can be easily described as follows.
Under the comultiplication $\Delta$, the elements $K_i, D$ are group-like elements and
\begin{align*}
\Delta(e_i)=e_i\ot 1+K_i\ot e_i, \qquad \Delta(f_i)=f_i\ot K_i^{-1}+1\ot f_i.
\end{align*}
The antipode $S$ is given by
$S(e_i)=-K_i^{-1}e_i, S(f_i)=-f_i\,K_i,
S(K_i)=K_i^{-1}, S(D)=D^{-1}$, and the counit $\vep$
is given by $\varepsilon(e_i)=\varepsilon(f_i)=0, \varepsilon(K_i)=\varepsilon(D)=1$.

%The following statement is an immediate consequence of the quantum PBW theorem.
%
%\begin{prop} As vector spaces
%$$U_{q}(\hat{\mathfrak{g}})\cong U_{q}(\widehat{\mathfrak n}^-)\ot U^0\ot
%U_{q}(\widehat{\mathfrak n}), $$
%where
%$U_{q}(\widehat{\mathfrak n})$ (resp.
%$U_{q}(\widehat{\mathfrak n}^-)$) is the subalgebra of
%$U_{q}(\hat{\mathfrak{g}})$ generated
%by $e_i$ (resp. $f_i$) for all $i\in I$ and
%$U^0=$ $\mathbb{C}[K_0^{\pm1},\cdots,K_n^{\pm1},\,\gamma^{\pm\frac{1}2},\,D^{\pm 1}]$.
%\end{prop}

The $v$-bracket is defined by $[a, b]_v=ab-vba$. For $v_i\in \mathbb{C}(q,\,q^{-1})\backslash \{0\}$,
we define two types of multi-brackets inductively by
%  the quantum Lie brackets $[\,a_1, a_2,\cdots,
%  a_s\,]_{(v_1, \cdots,\, v_{s-1})}$ and
%  $[\,a_1, a_2, \cdots,
%  a_s\,]_{\la v_1,\cdots, \,v_{s-1}\ra}$ are defined
 % inductively by
   \begin{gather*}   % [\,a_1, a_2\,]_{v_1}=a_1a_2-v_1\,a_2a_1,\\
  [\,a_1, a_2, \cdots, a_s\,]_{(v_1,\,v_2,\,\cdots,
  \,v_{s-1})}=[\,a_1, \cdots, [a_{s-1},
  a_s\,]_{v_1}]_{(v_2,\,\cdots,\,v_{s-1})},\\
 [\,a_1, a_2, \cdots, a_s\,]_{\la v_1,\,v_2,\,\cdots,
  \,v_{s-1}\ra}=[\,[\,a_1, a_2]_{v_1} \cdots, a_{s-1}\,]_{\la
  v_2,\,\cdots,\,v_{s-2}\ra}.
 \end{gather*}
The following identities will be useful \cite{J1}:
\begin{eqnarray}
%&&[\,a, bc\,]_v=[\,a, b\,]_q\,c+q\,b\,[\,a, c\,]_{\frac{v}q},\\
%&&[\,ab, c\,]_v=a\,[\,b, c\,]_q+q\,[\,a, c\,]_{\frac{v}q}\,b, \\
&& [\,a,[\,b,c\,]_u\,]_v=[\,[\,a,b\,]_q,
c\,]_{\frac{uv}q}+q\,[\,b,[\,a,c\,] _{\frac{v}q}\,]_{\frac{u}q},\label{v1}\\
&&[\,[\,a,b\,]_u,c\,]_v=[\,a,[\,b,c\,]_q\,]_{\frac{uv}q}+q\,[\,[\,a,c\,]
_{\frac{v}q},b\,]_{\frac{u}q}.\label{v2}
\end{eqnarray}

The quantum affine algebra also admits Drinfeld's new realization \cite{D1}, which is the quantum analog
of the current algebra realization for the affine Lie algebra.

\begin{defi} \cite{B, D2}
The Drinfeld realization
$\mathcal{U}_{q}(\hat{\mathfrak{g}})$ is an associative algebra over $\mathbb C(q)$
generated by $x_i^{\pm}(k)$, $a_i(r)$, $K_i^{\pm1}$,
 $D^{\pm1}$ $( i\in I_0$,
$k\in \mathbb{Z}$, $r\in \mathbb{Z}\backslash
\{0\})$ and the central elements $\gamma^{\pm\frac{1}{2}}$ satisfy the following relations.
\begin{align} \label{e:Dr1}
& D\,x_i^{\pm}(k)\,D^{-1}=q^k, \quad D\, a_i(r)\,D^{-1}=q^r\,a_i(r),\\ \label{e:Dr2}
&  K_i^{\pm 1}\,K_i^{\mp 1}=1, \quad K_i\,x_j^{\pm}(k)\, K_i^{-1} = q^{\pm a_{ij}}x_j^{\pm}(k),
\\  \label{e:Dr5}
& [\,a_i(r),~K_j^{{\pm }1}\,]=0,
 \\ \label{e:Dr3}
&[\,a_i(r),a_j(r')\,]
=\delta_{r+r',0}\frac{
[\,r a_{i,j}\,]_i}{r}\cdot\frac{\gamma^{r}-\gamma^{-r}}{q-q^{-1}},\\ \label{e:Dr6}
&[\,a_i(r),x_j^{\pm}(k)\,]=\pm \frac{ [r a_{ij}]_i}{r} \gamma^{\mp\frac{r}2}x_j^{\pm}(r{+}k),\\ \label{e:Dr7} %e:comm3
&[x_i^{\pm}(k+1),x_j^{\pm}(l)\,]_{q^{\pm(\alpha_i, \alpha_j)}}=-[x_j^{\pm}(l+1),x_i^{\pm}(k)\,]_{q^{\pm(\alpha_i, \alpha_j)}},\\ \label{e:Dr8}
&[\,x_i^{+}(k),~x_j^-(k')\,]=\frac{\delta_{ij}}{q_i-q_i^{-1}}\Big({\gamma}^{\frac{k-k'}{2}}\,
\phi_i(k{+}k')-\gamma^{\frac{k'-k}{2}}\,\varphi_i(k{+}k')\Big),%\leqno(\textrm{D8})
\end{align}
where $\phi_i(m)$, $\varphi_i(-m)~(m\in \mathbb{Z}_{\geq 0})$ such that
$\phi_i(0)=K_i$ and  $\varphi_i(0)=K_i^{-1}$ are defined as below:
\begin{align*} %\label{e:Dr9}
\sum\limits_{m=0}^{\infty}\phi_i(m) z^{-m}&=K_i \exp \Big(
(q_i{-}q_i^{-1})\sum\limits_{\ell=1}^{\infty}
 a_i(r)z^{-r}\Big),\\ %\label{e:Dr10}
\sum\limits_{m=0}^{\infty}\varphi_i(-m) z^{m}&=K_i^{-1} \exp
\Big({-}(q_i{-}q_i^{-1})
\sum\limits_{r=1}^{\infty}a_i(-r)z^{r}\Big),
\end{align*}
\begin{align}\label{e:Dr11}
& \underset{m_1,\cdots,
m_{n}}{Sym}
\sum_{k=0}^{n}(-1)^k
\Big[{n\atop  k}\Big]_{i}x_i^{\pm}(m_1)\cdots x_i^{\pm}(m_k) x_j^{\pm}(l)x_i^{\pm}(m_{k+1})\cdots x_i^{\pm}(m_{n})=0, \quad i\neq j,
\end{align} %\leqno{(\textrm{D8})}
 where $n=1-a_{ij}$, $[m]_{i}=\frac{q_i^{m}-q_i^{-m}}{q_i-q_i^{-1}}$, $[m]_{i}!=[m]_{i}\cdots[2]_{i}[1]_{i}$,
$\Bigl[{m\atop n}\Bigr]_{i}=\frac{[m]_{i}!}{[n]_{i}![m-n]_{i}!}$, and
$\textit{Sym}_{m_1,\cdots, m_n}$  denotes the symmetrization with respect to the indices $(m_1, \cdots, m_n)$.
\end{defi}

%The following well-known result follows easily from the definition and commutation relations
%by rearranging monomial elements of the algebra in the prescribed order.
%
%\begin{prop}$\mathcal{U}_{q}(\hat{\mathfrak{g}})$ has a triangular
%decomposition:
%$$\mathcal{U}_{q}(\hat{\mathfrak{g}})\cong
%\mathcal{U}_{q}(\widetilde{\mathfrak{n}}^-)\otimes\mathcal{U}_{q}^0(\widehat{\mathfrak{g}})
%\otimes\mathcal{U}_{q}(\widetilde{\mathfrak{n}}^+),$$  where
%$\mathcal{U}_{q}(\widetilde{\mathfrak{n}}^\pm)$ is
%generated by $x_i^\pm(k)$ respectively, and
%$\mathcal{U}_{q}^0(\widehat{\mathfrak{g}})$ is
%generated by $K_i^{\pm1}$,
%$\gamma^{\pm\frac1{2}}$,  $D^{\pm1}$,
% and $a_i(\pm\ell)$.
%\end{prop}

We now choose a convenient set of generators for $\mathcal{U}_{q}(\hat{\mathfrak{g}})$:
%Denote by $\mathcal{U}'_{q}(\hat{\mathfrak{g}})$ the subalgebra of
%$\mathcal{U}_{q}(\hat{\mathfrak{g}})$ generated by
\begin{align}\label{e:generators}
x_i^{\pm}(0),\,x_1^+(-1), x_1^-(1), K_i^{\pm1} ,
\gamma^{\pm\frac{1}2}, \quad D^{\pm1} \quad i\in I_0,
\end{align}
where $\gamma^{\pm\frac 12}$ are central and the relations are: \eqref{e:Dr2}-\eqref{e:Dr5} for
$l=0, k=0, \pm 1$; \eqref{e:Dr8} for $k=k'=0$ and $k=-k'=1$; and
 \begin{align}
 %\label{sub:1}
%	&\gamma^{\pm\frac{1}{2}}~~
%	\textrm{are central and} \quad K_i^{\pm 1}\,K_i^{\mp 1}=1,
%	\\ \label{sub:2}
%&D\,x_i^{\pm}(k)\,D^{-1}=q^k\, x_i^{\pm}(k),\\	
%&K_i\,x_{j}^{\pm}(0)\, K_i^{-1} = q_i^{\pm a_{i{j}}}x_{j}^{\pm}(0), \qquad K_i\,x_{1}^{\pm}(\mp 1)\, K_i^{-1} = q_i^{\pm a_{{i}1}}x_{1}^{\pm}(\mp 1),	\\
\label{sub:3}
&[\,x_1^{\epsilon}(-\epsilon ),\,x_1^{\epsilon}(0)\,]_{q^{-2}}=0,
%\\\label{sub:4}	
%&[\,x_i^{+}(0),\,x_j^{-}(0)\,]=\delta_{ij}\frac{K_i-K_i^{-1}}{q-q^{-1}},\qquad [\,x_1^{+}(-1),\,x_1^{-}(1)\,]=\frac{\gamma^{-1}K_1-\gamma K_1^{-1}}{q-q^{-1}},
\\\label{sub:5}
&[\,x_i^{-\epsilon}(0),\,x_1^{\epsilon}(-\epsilon)\,]=0,
\quad \hbox{for} \quad i\neq 0,\\\label{sub:6}
&[x_i^{\pm}(0), x_j^{\pm}(0)]=0, \qquad [x_1^{\epsilon}(-\epsilon), x_k^{\epsilon}(0)]=0, \quad\mbox{for}\ \  a_{ij}=a_{1k}=0,\\\label{sub:7}
&  \sum_{k=0}^{n=1-a_{i{j}}}(-1)^k
	\Big[{n\atop  k}\Big]_{i}(x_i^{\pm}(0))^k x_{j}^{\pm}(0)(x_i^{\pm}(0))^{n-k}=0,
	\quad\hbox{for} \quad   i\neq j,\\\label{sub:8}
& \sum_{k=0}^{2}(-1)^k
	\Big[{2\atop  k}\Big]_{i}(x_2^{\epsilon}(0))^k x_1^{\epsilon}(-\epsilon)(x_2^{\epsilon}(0))^{2-k}=0,\\\label{sub:9}
 &  \sum_{k=0}^{2}(-1)^k
	\Big[{2\atop  k}\Big]_{i}(x_1^{\epsilon}(-\epsilon))^k x_{2}^{\epsilon}(0)(x_1^{\epsilon}(-\epsilon))^{2-k}=0,\\
 \label{sub:10}
 &  Sym_{{0},
		{-\epsilon}}\sum_{k=0}^{2}(-1)^k
	\Big[{2\atop  k}\Big]_{i}(x_1^{\epsilon}(-\epsilon))^k x_{2}^{\epsilon}(0)(x_1^{\epsilon}(0))^{2-k}=0,
	\end{align}
	where $\textit{Sym}_{{m_1},
		{m_{2}}}$  denotes the symmetrization with respect to the indices $({m_1}, {m_{2}})$.

% which are same to some relations $(\ref{e:Dr1})$ to $(\ref{e:Dr11})$ (except those involving
%$a_i(l)$) for the indices
%allowed in the set of the generators.
%\begin{equation*}
%{\mathcal U}'_{q}(\hat{\mathfrak{g}})=\langle x_i^{\pm}(0), x_1^+(-1), x_1^-(1), K_i^{\pm1},
%\gamma^{\pm\frac{1}2}, D^{\pm1}|i\in I\rangle.
%\end{equation*}

Starting from the elements in (\ref{e:generators}), we can inductively determine other Drinfeld generators as follows.
First we have
\begin{align}
a_1(\pm 1)=\pm K_1^{\mp 1}\gamma^{1/2}[ x_1^{\pm}(0), x^{\mp}_1(\pm 1)].
%a_1(-1)&=K_1\gamma^{1/2}[x^+_1(-1), x_1^-(0)].
\end{align}
Then $x_2^{+}(\mp 1)$ and $x_2^{-}(\mp 1)$ can be determined by relation $(\ref{e:Dr6})$, and subsequently all other $x_i^{+}(\mp 1)$ and $x_i^{-}(\mp 1)$
are fixed. Using this idea, one can inductively prove the following result
(cf. \cite{JZ3}).
%the result was proved for more general two-parameter quantum affine algebras.
\begin{prop} %\cite{JZ3}\,
As an associative algebra, the algebra generated by
the elements in (\ref{e:generators}) subject to the aforementioned relations is identical to
${\mathcal U}_{q}(\hat{\mathfrak{g}})$.
\end{prop}

%In the next section, we will give another
%proof in the general setting.
%\begin{proof}\, The proof is a special case of the two-parameter case in \cite{JZ3}.
%\end{proof}

%To simplify presentation, we use the following convention for $q$-brackets \cite{J1}.

We can describe the isomorphism between the two presentations in details.

Let $X_{\theta}=[e_{i_{h-1}}, [e_{i_{h-2}}, \cdots, [e_{i_2}, e_{i_1}]\cdots ]$ be the
root vector associated with the maximum root
$\theta=\alpha_{i_{h-1}}+\cdots+\alpha_{i_2}+\alpha_{i_1}$ of $\mathfrak g$.
This determines a sequence:
\begin{equation}\label{L:reduced}
i_1, i_2, \cdots, i_{h-1},
\end{equation}
where $i_k\in I$ and $h$ is the Coxeter number. We will call such a sequence a
{\it root chain} to the maximum root. Clearly root chains of the maximum root are not unique.

With respect to such a root chain for the maximum root $\theta$, we define the following numbers: for $t=2, \ldots, h-1$,
\begin{align*}
%u_{j}&=q^{(\alpha_{i_j},\alpha_{i_{j-1}})+(\alpha_{i_j},\alpha_{i_{j-2}})+\cdots+(\alpha_{i_j},\alpha_{i_{1}})}
%,\\
%u'_{j}&=q^{(\alpha_{i_j},\alpha_{i_{j}})+(\alpha_{i_j},\alpha_{i_{j-1}})+\cdots+(\alpha_{i_j},\alpha_{i_{1}})}
% ,\\
%u''_{j}&=q^{(\alpha_{i_j},\alpha_{i_{j+1}})+(\alpha_{i_j},\alpha_{i_{j}})+\cdots+(\alpha_{i_j},\alpha_{i_{1}})}
% ,\\
v_{j_t}&=q^{-((\alpha_{j_1},\alpha_{j_{1}})+(\alpha_{j_1},\alpha_{j_{2}})+\cdots+(\alpha_{j_{1}},\alpha_{j_{t}}))}
 ,\\
v'_{j_t}&=q^{-((\alpha_{j_1},\alpha_{j_{t}})+(\alpha_{j_2},\alpha_{j_{t}})+\cdots+(\alpha_{j_{t}},\alpha_{j_{t}}))}
 ,\\
v''_{j_t}&=q^{-((\alpha_{j_1},\alpha_{j_{t}})+(\alpha_{j_2},\alpha_{j_{t}})+\cdots+(\alpha_{j_{t+1}},\alpha_{j_{t}}))}.
\end{align*}

\begin{theorem}\cite{JZ3} For each fixed root chain to the maximum root
$\theta$ given in (\ref{L:reduced}), the map $\Psi:
U_{q}(\hat{\mathfrak{g}}) \longrightarrow {\mathcal
U}_{q}(\hat{\mathfrak{g}})$ defined below is an algebra isomorphism:
\begin{align*}
&K_i\longmapsto K_i, \quad\qquad
 K_0\longmapsto \gamma\,K_{\theta}^{-1},\\
&\gamma^{\pm\frac{1}2}\longmapsto \gamma^{\pm\frac{1}2},\qquad
D^{\pm1}\longmapsto D^{\pm1},\\
&e_i\longmapsto x_i^+(0),\qquad e_0\longmapsto  x^-_{\theta}(1)\cdot(\gamma\,K_{\theta}^{-1}),\\
&f_i\longmapsto x_i^-(0),\qquad
f_0 \longmapsto (\gamma^{-1}\,{K}_{\theta})\cdot x_{\theta}^+(-1),
\end{align*}
where $K_{\theta}=K_{i_1}\,\cdots\,
K_{i_{h-1}}$ and
\begin{align}\label{e:invmap1}
x_{\theta}^-(1)&=[x_{i_{h-1}}^-(0),\,\cdots,\, x_{i_2}^-(0),\, x_{i_1}^-(1)]_{(v_{i_2},\, \cdots,\,v_{i_{h-1}})}, \\ \label{e:invmap2} x_{\theta}^+(-1)&=[x_{i_{1}}^+(-1),\, x_{i_2}^+(0),\,\cdots,\,x_{i_{h-1}}^+(0)]_{\langle v_{i_2}^{-1},\, \cdots,\,v_{i_{h-1}}^{-1}\rangle}.
\end{align}
\end{theorem}

For example in type $A$, \eqref{e:invmap1}-\eqref{e:invmap2} read as follows: % for the root chain (\ref{L:reduced}):
\begin{align*}
x_{\theta}^-(1)&=[x_{n-1}^-(0),\,\cdots,\, x_{2}^-(0),\, x_{1}^-(1)]_{(q^{-1},\, \cdots,\,q^{-1})}, \\ x_{\theta}^+(-1)&=[x_{1}^+(-1),\, x_{2}^+(0),\,\cdots,\,x_{n-1}^+(0)]_{\langle q,\, \cdots,\,q\rangle}.
\end{align*}

%Given the reduced expression of $w_0$ we define the following numbers: for $j=2, \ldots, h-1$
%\begin{align*}
%%u_{j}&=q^{(\alpha_{i_j},\alpha_{i_{j-1}})+(\alpha_{i_j},\alpha_{i_{j-2}})+\cdots+(\alpha_{i_j},\alpha_{i_{1}})}
%%,\\
%%u'_{j}&=q^{(\alpha_{i_j},\alpha_{i_{j}})+(\alpha_{i_j},\alpha_{i_{j-1}})+\cdots+(\alpha_{i_j},\alpha_{i_{1}})}
%% ,\\
%%u''_{j}&=q^{(\alpha_{i_j},\alpha_{i_{j+1}})+(\alpha_{i_j},\alpha_{i_{j}})+\cdots+(\alpha_{i_j},\alpha_{i_{1}})}
%% ,\\
%v_{j}&=q^{-((\alpha_{i_1},\alpha_{i_{j}})+(\alpha_{i_2},\alpha_{i_{j}})+\cdots+(\alpha_{i_{j-1}},\alpha_{i_{j}}))}
% ,\\
%v'_{j}&=q^{-((\alpha_{i_1},\alpha_{i_{j}})+(\alpha_{i_2},\alpha_{i_{j}})+\cdots+(\alpha_{i_{j}},\alpha_{i_{j}}))}
% ,\\
%v''_{j}&=q^{-((\alpha_{i_1},\alpha_{i_{j}})+(\alpha_{i_2},\alpha_{i_{j}})+\cdots+(\alpha_{i_{j+1}},\alpha_{i_{j}}))}.
%\end{align*}
%We also denote $t'_{j}=\frac{v_{j}-v_{j}^{-1}}{q_{i_j}-q_{i_j}^{-1}}$. %{\color{red} Is it $u_i=v_i^{-1}$?}

This theorem was proved in \cite{JZ3}
for the two-parameter quantum affine algebra using the new generators, which include the above as a special case.
%thus also shows the Drinfeld isomorphism for one-parameter case.
In particular,
the inverse isomorphism of $\Phi$ can be given as follows.
%In the following we outline the
% steps for later purpose. %\cite{JZ3}.

%\begin{theo}\, $\Psi:\,U_{q}(\hat{\mathfrak{g}}) \longrightarrow {\mathcal U}_{q}(\hat{\mathfrak{g}})$ is an algebra homomorphism.
%\end{theo}

\begin{theo} \cite{JZ3} \,Let $1=i_1,\,i_2,\, \ldots, i_{h-1}$ be a fixed root chain to the maximum root
$\theta$ given in (\ref{L:reduced}). Then the following map $\Phi:\,{\mathcal U}_{q}(\hat{\mathfrak{g}})\longrightarrow U_{q}(\hat{\mathfrak{g}})$ given below is an epimorphism, where for $\forall i \in I$
\begin{align*}
&\Phi(\gamma)=\gamma,\qquad \Phi(D)=D,\qquad \Phi(K_i)=K_i,\\
&\Phi(x_i^-(0))=f_i,\qquad \Phi(x_i^+(0))=e_i ,\\
&\Phi(x_1^-(1))=[\,e_{i_2},\,e_{i_3},\,\cdots,\, e_{i_{h-1}},\, e_0\,]_{({v'}_{i_{h-1}}^{-1},\, \cdots,\, {v'}_{i_2}^{-1})}\gamma^{-1}K_{1},\\
&\Phi(x_1^+(-1))=\gamma K_{1}^{-1}[\,f_0,\,f_{i_{h-1}},\,f_{i_{h-2}},\,\cdots,\, f_{i_{2}}\,]_{\langle v'_{i_{h-1}},\, \cdots,\, v'_{i_2}\rangle}.
\end{align*}
\end{theo}

 We can now describe the Hopf structure of $\mathcal{U}_{q}(\hat{\mathfrak{g}})$ in terms of the simplified (Drinfeld) generators.

For $i_2\leq j_1<\cdots<j_l\leq i_{h-1}$, define $x_{1,\bold j}^-(1)$ and $x_{1,\bold j}^+(-1)$ inductively by
\begin{align*}
&x_{1\,\bold{j}}^-(1)=[\,x_{j_l}^-(0),\,\cdots, x_{j_1}^-(0),\,x_{1}^-(1)\,]_{(v_{j_1},\, \cdots,\, v_{j_l})},\\
%&x_{1\,l}^-(1)=[\,x_{j_l}^-(0),\,\cdots, x_{j_1}^-(0),\,x_{1}^-(1)\,]_{(v_{j_1},\, \cdots,\, v_{j_l})},\\
&x_{1\,\bold{j}}^+(-1)=[\,x_{1}^+(-1),\, x_{j_1}^+(0),\,\cdots ,\,x_{j_{l}}^+(0)\,]_{\langle v_{j_1}^{-1},\, \cdots,\, v_{j_l}^{-1}\rangle}.
%&x_{1\,l}^+(-1)=[\,x_{1}^+(-1),\, x_{j_1}^+(0),\,\cdots ,\,x_{j_{l}}^+(0)\,]_{\langle v_{j_1}^{-1},\, \cdots,\, v_{j_l}^{-1}\rangle}.
\end{align*}

Similarly we also define $x_{\bold j}^+(0)$ and $x_{\bold j}^-(0)$ for $i_2\leq j_1<\cdots<j_l\leq i_{h-1}$ as follows.
\begin{align*}
&x_{\bold j}^+(0)=[\,x_{j_1}^+(0),\, x_{j_2}^+(0),\,\cdots ,\,x_{j_{l}}^+(0)\,]_{({v'}_{j_{l-1}}^{-1},\, \cdots,\, {v'}_{j_1}^{-1})},\\
&x_{\bold j}^-(0)=[\,x_{j_l}^-(0),\, x_{j_{l-1}}^-(0),\,\cdots ,\,x_{j_{1}}^-(0)\,]_{\langle {v'}_{j_{l-1}},\, \cdots,\, {v'}_{j_1}\rangle}.
%&x_{j_1\,l}^+(0)=[\,x_{j_1}^+(0),\, x_{j_2}^+(0),\,\cdots ,\,x_{j_{l}}^+(0)\,]_{({v'}_{j_{l-1}}^{-1},\, \cdots,\, {v'}_{j_1}^{-1})},\\
%&x_{j_1,\,l}^-(0)=[\,x_{j_l}^-(0),\, x_{j_{l-1}}^-(0),\,\cdots ,\,x_{j_{1}}^-(0)\,]_{\langle {v'}_{j_{l-1}},\, \cdots,\, {v'}_{j_1}\rangle}.
\end{align*}
%where  $p_j=q^{(\alpha_{i_j}, \alpha_{i_{j+1}})}$, $j=2, \cdots, h-1$. % and $v_{i_j}=v'_ju''_j$.
%Now we can describe the comultiplication on the Drinfeld generators of $\mathcal{U}(\hat{\mathfrak{g}})$.

As before, we fixed the root chain $1=i_1,\,i_2,\, \ldots, i_{h-1}$ to the maximum root
$\theta$ given in (\ref{L:reduced}).
The comultiplication
$\Delta: \mathcal{U}_{q}(\hat{\mathfrak{g}})\longrightarrow \mathcal{U}_{q}(\hat{\mathfrak{g}})\otimes
\mathcal{U}_{q}(\hat{\mathfrak{g}})$ is explicitly given as follows:  for $i\in I$
\begin{gather*}
\Delta(K_i)=K_i\ot K_i, \quad
\Delta(\gamma^{\pm\frac{1}2})=\gamma^{\pm\frac{1}2}\otimes\gamma^{\pm\frac{1}2}, \quad
\Delta(D^{\pm1})=D^{\pm1}\otimes D^{\pm1},\\
\Delta(x_i^+(0))=x_i^+(0)\ot 1+K_i\ot x_i^+(0),\\
 \Delta(x_i^-(0))=x_i^-(0)\ot K_i^{-1}+1\ot
x_i^-(0),\\
\Delta(x_1^-(1))=x_1^-(1)\ot \gamma^{-1}K_1+1\ot x_1^-(1)+\sum_{i_2\leq j_1<\cdots<j_l\leq i_{h-1}}\xi_{\bold j}\, x_{1\,\bold j}^-(1)\ot x_{\bold j}^+(0)\gamma^{-1}K_1,\\
\Delta(x_1^+(-1))=x_1^+(-1)\ot 1+\gamma K_1^{-1}\ot x_1^+(-1)+\sum_{i_2\leq j_1<\cdots<j_l\leq i_{h-1}}\zeta_{\bold j}\, \gamma K_1^{-1}\, x_{\bold j}^-(0)\ot x_{1,\bold j}^+(-1),
\end{gather*}
where
\begin{align*}
\xi_{\bold j}&=(v'_{j_l}-{v'}^{-1}_{j_l})t_{j_2}^{-1}\cdots t_{j_l}^{-1}{v''}_{j_2}\cdots {v''}_{j_{l-1}},\\
\zeta_{\bold j}&=({v'}^{-1}_{j_l}-v'_{j_l})t_{j_2}^{-1}\cdots t_{j_l}^{-1}{v''}^{-1}_{j_2}\cdots {v''}^{-1}_{j_{l-1}},
\end{align*}
and $t'_{j_k}=\frac{v_{j_k}-v_{j_k}^{-1}}{q_{j_k}-q_{j_k}^{-1}}$. These formulas extend the action of the
original Drinfeld-Jimbo comultiplication from the Chevelley generators to Drinfeld's new generators, so
they are different from Drinfeld's infinite coproduct formulas (see \cite{DF, J2, G}).

\begin{prop}\, %Let $1=i_1,\,i_2,\, \cdots, i_{h-1}$ be a sequence of indices specified above.
The algebra $\mathcal{U}_{q}(\hat{\mathfrak{g}})$  is a Hopf algebra with the comultiplication $\Delta$ (defined above), counit $\vep$ and antipode $S$ given as follows: for $i\in I$
\begin{eqnarray*}
&&\varepsilon(x_i^+(0))=\varepsilon(x_i^-(0))=\varepsilon(x_i^+(-1))=\varepsilon(x_i^-(1))=0,\\
&& \varepsilon(K_i)=\varepsilon(D)=1,\quad S(K_i)=K_i^{-1}, \quad  S(D)=D^{-1}, \\
&&S(x_i^+(0))=-K_i^{-1}x_i^+(0),\qquad S(x_i^-(0))=-x_i^-(0)\,K_i,\\
&&S(x_1^+(-1))=-\gamma^{-1}{K_1}x_1^+(-1)-\sum_{i_2\leq j_1<\cdots<j_l\leq i_{h-1}}\zeta_{\bold j}\,\, y_{\bold j}^-(0)\gamma^{-1}{K_{j_1}}\cdots {K_{j_l}}\,x_{1,\bold j}^+(-1),\\
&&S(x_1^-(1))=-x_1^-(1){\gamma}{K_1}-\sum_{i_2\leq j_1<\cdots<j_l\leq i_{h-1}}\xi_{\bold j}\,\, x_{1\bold j}^-(1){\gamma}{K_{j_1}}^{-1}\cdots {K_{j_l}}^{-1}\, y_{\bold j}^+(0),
\end{eqnarray*}
where the sums run over all sequences $\bold j=(j_1,\cdots, j_l)$ such that $i_2\leq j_1<\cdots<j_l\leq i_{h-1}$ and
\begin{align*}
y_{\bold j}^-(0)&=a\,[\,x_{j_l}^-(0),\, x_{j_{l-1}}^-(0),\,\cdots ,\,x_{j_{1}}^-(0)\,]_{\langle p_{j_{l-1}},\, \cdots,\, p_{j_1}\rangle},\\
y_{\bold j}^+(0)&=a^{-1}\,[\,x_{j_1}^+(0),\, x_{j_2}^+(0),\,\cdots ,\,x_{j_{l}}^+(0)\,]_{({p}^{-1}_{j_{l-1}},\, \cdots,\, {p}^{-1}_{j_1})},
\end{align*}
%$$y_{j_1,\,l}^+(0)=a^{-1}\,[\,x_{j_1}^+(0),\, x_{j_2}^+(0),\,\cdots ,\,x_{j_{l}}^+(0)\,]_{({p}^{-1}_{{l-1}},\, \cdots,\, {p}^{-1}_{1})},$$
and for $1\leqslant t \leqslant l-1$, $p_{j_{t}}=q^{(\alpha_{j_t},\,\alpha_{j_{t+1}})}$. The constant $a$ is given by
$$a=\prod\limits_{t=1}^{l-1}q^{2(\alpha_{j_t}, \alpha_{j_{t+1}})+(\alpha_{j_t}, \alpha_{j_{t+2}})+\cdots+(\alpha_{j_t}, \alpha_{j_l})}.$$
\end{prop}

\begin{theo}\, The morphisms $\Phi$ and $\Psi$ are two coalgebra homomorphisms:
 $$\Delta \circ \Psi=(\Psi\otimes \Psi)\Delta, \qquad\Delta \circ \Phi=(\Phi\otimes \Phi)\Delta.$$
 Moreover $\Phi$ and $\Psi$ are mutually inverse to each other.
 %, thus $\mathcal{U}_{q}(\hat{\mathfrak{g}})$ and $U_{q}(\hat{\mathfrak{g}})$ are isomorphic Hopf algebras.
\end{theo}
\begin{remark} This result
upgrades the algebra isomorphism into a Hopf algebra isomorphism (cf. \cite{D1}) between the Drinfeld-Jimbo definition and the Drinfeld realization. In viewing of \cite{GM, GTL} it also provides a comultiplication formula for
the Yangian algebra associated to any simple Lie algebra in terms of the Drinfeld generators. The main terms of
the comultiplication formula were known for untwisted quantum affine algebras
in \cite[Th. 2.3]{JKK} and for twisted quantum affine algebras
in \cite[Th. 2.2]{JM}.

As our comultiplication is determined by the usual Drinfeld-Jimbo comultiplication, it is compatible with
the $q$-characters \cite{FR1, FM}.
 \end{remark}

\section{Quantum toroidal algebra $U_q(sl_{n+1},tor)$}
We now generalize the Hopf algebra structure to the (one-parameter) quantum toroidal algebra. We focus on
type $A$ \cite{GKV, VV}, though similar results
are expected for ADE types \cite{FJW} and the two-parameter case %quantum toroidal algebras
\cite{JZ1}. However, our construction seems not work for the
quantum toroidal algebra associated to $gl(1)$ (\cite{FJM1, FJM2, MO}).

Fix the integer $n\geq 1$, and let $q$ be a generic complex number. The quantum toroidal algebra $U_q(sl_{n+1},tor)$ is an associative algebra over $\mathbb C(q)$ with generators $x_i^{\pm}(k),\, a_i(r),\, K_i^{\pm}$, $\gamma^{\pm\frac{1}{2}},
q^{\pm d_1},$   $q^{\pm d_2}$, $(i\in I,\, k\in\mathbb{Z},\,r\in \mathbb{Z}/\{0\})$ satisfying relations
\eqref{e:Dr3}, \eqref{e:Dr6}, \eqref{e:Dr7}, \eqref{e:Dr8}, %\eqref{e:Dr9}, \eqref{e:Dr10},
\eqref{e:Dr11}
and the following additional relations:
\begin{align}\label{e:comm0} %3.1
&\gamma^{\pm\frac{1}2} \textrm{are central such that,}\ \gamma^{\pm \frac{1}{2}}\gamma^{\mp \frac{1}{2}}=1,\,
  K_i^{\pm1}, a_j(r), \textrm{and $q^{\pm d_i}$ commute among each other},\\ \label{e:comm00} %3.2
%&q^{\pm d_i}q^{\mp d_j}=q^{\mp d_j}q^{\pm d_i},\\ \label{e:comm0a} %3.3
%&\textrm{$q^{\pm d_i}, K_i^{\pm}$ commute with each other},\\
&K_ix_j^{\pm}(k)K_i^{-1}=q^{\pm a_{ij}}x_j^{\pm}(k),\\\label{e:comm0b}
&q^{d_1}a_i(r)q^{-d_1}=q^ra_i(r), \quad q^{d_1}x^{\pm}_i(k)q^{-d_1}=q^kx^{\pm}_i(k),\\ \label{e:comm0c}
&q^{d_2}a_i(r)q^{-d_2}=a_i(r), \quad q^{d_2}x^{\pm}_i(k)q^{-d_2}=q^{\delta_{i0}}x^{\pm}_i(k).
\end{align}
%\label{e:comm1} %3.6{e:Dr3}
%&[\,a_i(r),a_j(s)\,]
%=\delta_{r+s,0}\frac{[\,r\,a_{ij}\,]}{r}
%\frac{\gamma^{r}-\gamma^{-r}}{q-q^{-1}},\\ \label{e:comm2} %3.7{e:Dr6}
%&[\,a_i(r),x_j^{\pm}(k)\,]=\pm \frac{[\,r\,a_{ij}\,]}{r}
%\gamma^{\mp\frac{|r|}{2}}x_j^{\pm}(r{+}k),\\ \label{e:comm3} %3.8{e:Dr7}
%&[\,x_i^{\pm}(k+1),\,x_j^{\pm}(l)\,]_{q^{\pm (\alpha_i, \alpha_j)}}+
%[\,x_j^{\pm}(l+1),\,x_i^{\pm}(k)\,]_{q^{\pm (\alpha_i, \alpha_j)}}=0, \\ \label{e:comm4} %3.9{e:Dr8}
%&[\,x_i^{+}(k),\,x_j^{-}(l)\,]=\delta_{ij}\frac{\gamma^{\frac{k-l}{2}}\phi_i(k+l)-\gamma^{\frac{l-k}{2}}\varphi_i(k+l)}{q-q^{-1}},
%\end{align}
%where $\phi_i(r)$, and $\varphi_i(-r)\, (r\geq 0)$ %such that $\phi_i(0)=K_i$ and  $\varphi_i(0)=K_i^{-1}$
%are defined by:
%\begin{gather*}\sum\limits_{m=0}^{\infty}\phi_i(r) z^{-r}=K_i \exp \Big(
%(q{-}q^{-1})\sum\limits_{\ell=1}^{\infty}
% a_i(r)z^{-\ell}\Big), \\
%\sum\limits_{r=0}^{\infty}\varphi_i(-r) z^{r}=K_i^{-1}\exp
%\Big({-}(q{-}q^{-1})
%\sum\limits_{\ell=1}^{\infty}a_i(-\ell)z^{\ell}\Big),
%\end{gather*}
%\begin{align}\label{e:comm5}       %{e:Dr11} just for a_{ij}=-1
%x_i^{\pm}(k_1)x_i^{\pm}(k_2)x_j^{\pm}(l)-[2]x_i^{\pm}(k_1)x_j^{\pm}(l)x_i^{\pm}(k_2)+
%x_j^{\pm}(l)x_i^{\pm}(k_1)x_i^{\pm}(k_2)\\
%\hspace{5cm} +(k_1\leftrightarrow k_2)=0,\quad \textit{for} \quad a_{ij}=-1. \nonumber
%\end{align}
%Here $[m]=\frac{q^m-q^{-m}}{q-q^{-1}}$, %[x,y]_q=xy-qyx$,
%$A=(a_{ij})$ is the Cartan matrix of type $A_n^{(1)}$.

Let $\tau$ be the map of $\mathbb Z_{n+1}$ given by $\tau(i)=i+1$.
Then $\tau$ induces the diagram automorphism for the derived quantum toroidal algebra
$U_q'(sl_{n+1},tor)$, the subalgebra without the generator $q^{d_2}$, such that %\longrightarrow U_q'(sl_{n+1},tor)$ given by
\begin{align}
\tau(x_i^{\pm}(k))=x_{i+1}^{\pm}(k), \quad \tau(a_i(k))=a_{i+1}(k).
\end{align}

Let
$U_{q}({\mathfrak{n}}^\pm)$ be the subalgebra of $U_q(sl_{n+1},tor)$
generated by $x_i^\pm(k)$ ($i\in I$) respectively, and
$U_{q}^0({\mathfrak{g}})$ be the subalgebra of $U_q(sl_{n+1},tor)$
generated by $K_i^{\pm1}$, $\gamma^{\pm\frac1{2}}, q^{\pm d_1}, q^{\pm d_2}$ and $a_i(\pm r)$ for $i\in I$, $r\in \mathbb{N}$.

Using Lemma 3.3 of \cite{M1}, it is easy to see that as a vector space (cf. \cite{M2})
$$U_q(sl_{n+1},tor)\simeq
U_{q}({\mathfrak{n}}^-)\otimes U_{q}^0({\mathfrak{g}})
\otimes U_{q}({\mathfrak{n}}^{+}).$$

\section{The algebra $\mathcal{U}_0$}

%We first discuss a general method to extract a smaller set
%of generators for an associative algebra with generators and relations.
We denote by $\langle x_i\rangle$ the free associative algebra generated by $x_i$. We first introduce
the algebra $\mathcal{U}_0$ to
%The ground field is assumed to be the field $\mathbb C(q)$. The following result gives a method to
extract a smaller set of generators from  the quantum toroidal algebra. %a given associative algebra.
%\begin{lemm}\label{L:iso} Suppose $\mathcal A=\langlex_i\rangle/(R_1) $, $\mathcal B=\langle x_i, y_j\rangle/(R_1, R_2, R_{12})$ are two
%quotient algebras of the free associative algebras with respective relations $R_1=R_1(x_i), R_2=R_2(y_j), R_{12}=R_{12}(x_i, y_j)$.
%If $y_j=\langle x_i \rangle$ insider $\mathcal B$, and $R_2(y_j)\subset ( R_1)$, $R_{12}(x_i, y_j)\subset (R_1)$,
%then $\mathcal A\simeq \mathcal B$ as associative algebras.
%\end{lemm}
%\begin{proof} We define the map $\Phi: \mathcal A=\langlex_i\rangle/(R_1) \longrightarrow \mathcal B=\langle x_i, y_j\rangle/(R_1, R_2, R_{12})$ by
%sending $x_i\mapsto x_i$. Clearly $\Phi$ is an algebra homomorphism.
%
%On the other hand, we define the map $\Psi: \mathcal B=\langle x_i, y_j\rangle/(R_1, R_2, R_{12})\longrightarrow \mathcal A=\langle x_i\rangle/(R_1)$ by
%sending $x_i\mapsto x_i$, $y_j\mapsto \langle x_i \rangle$. Then $\Psi$ is also an algebra homomorphism.
%
%Clearly $\Psi\Phi(x_i)=\Psi(x_i)=x_i$, so $\Psi\Phi=Id$. As for the other direction, it is clear that $\Phi\Psi(x_i)=\Phi(x_i)=x_i$.
%Also $\Phi\Psi(y_j)=\Phi(\langle x_i \rangle)=\langle x_i \rangle=y_j$, hence $\Phi=\Psi^{-1}$.
%\end{proof}
%
%We consider $\{ x_i, y_j\}$ to the set of the whole Drinfeld generators in the quantum toroidal algebra, and now define
%a simplified set of generators $\{x_i\}$ in the following key definition.

\begin{defi} \label{4.2}The algebra $\mathcal{U}_0$ is an associative algebra over $\mathbb C(q)$ generated by
$x_i^{\pm}(0),\,x_0^+(-1)$, $x_0^-(1)$, $K_i^{\pm1}$ ($i\in I$), $q^{\pm d_1}, q^{\pm d_2}$
 and $\gamma^{\pm\frac{1}2}$ satisfying the following relations \eqref{f:comm0}-\eqref{f:comm7}, i.e.
$$
{\mathcal U}_0:=\left.\left\langle\, x_i^{\pm}(0), x_0^+(-1),\,x_0^-(1),\,K_i^{\pm1},\,q^{\pm d_1},\, q^{\pm d_2},\,
\gamma^{\pm\frac{1}2}\; \right| i\in I\;\right\rangle/\sim.
$$
%Here we only list the key relations just involving in the above finite generators for our purpose.
\begin{align}\label{f:comm0} %\label{f:comm00} %3.2 %3.1
&\gamma^{\pm\frac{1}2} \textrm{are central such that,}\ \gamma^{\pm \frac{1}{2}}\gamma^{\mp \frac{1}{2}}=1,\,  K_i^{\pm1}, q^{\pm d_i}  %[\,K_i^{\pm 1},a_j(r)\,]=0,
\textrm{and $K_i^{\pm}$ commute with each other},\\
%&q^{\pm d_i}q^{\mp d_j}=q^{\mp d_j}q^{\pm d_i}, \ \ i, j=1, 2,\\ \label{f:comm0a} %3.3
%&\textrm{$q^{\pm d_i}, K_i^{\pm}$ commute with each other},\\
&K_ix_j^{\pm}(0)K_i^{-1}=q^{\pm a_{ij}}x_j^{\pm}(0),\quad K_ix_0^{\epsilon}(-\epsilon)K_i^{-1}=q^{\pm a_{i0}}x_0^{\epsilon}(-\epsilon),\\
&%q^{d_1}a_i(r)q^{-d_1}=q^ra_i(r), \quad
q^{d_1}x^{\pm}_i(0)q^{-d_1}=x^{\pm}_i(0),\quad q^{d_1}x^{\epsilon}_0(-\epsilon)q^{-d_1}=q^{-\epsilon}x^{\epsilon}_0(-\epsilon),\\
&%q^{d_2}a_i(r)q^{-d_2}=a_i(r), \quad
q^{d_2}x^{\pm}_i(0)q^{-d_2}=q^{\delta_{i0}}x^{\pm}_i(0),\quad q^{d_2}x^{\epsilon}_0(-\epsilon)q^{-d_2}=qx^{\epsilon}_0(-\epsilon),\\ %\label{f:comm1} %3.6
%&[\,a_i(r),a_j(s)\,]
%=\delta_{r+s,0}\frac{[\,r\,a_{ij}\,]}{r}
%\frac{\gamma^{r}-\gamma^{-r}}{q-q^{-1}}p^{rb_{ij}},\\ \label{e:comm2} %3.7
%&[\,a_i(r),x_j^{\pm}(k)\,]=\pm \frac{[\,r\,a_{ij}\,]}{r}
%\gamma^{\mp\frac{|r|}{2}}x_j^{\pm}(r{+}k)p^{rb_{ij}},\\
\label{f:comm3} %3.8
&[\,x_0^{\epsilon}(-\epsilon),\,x_0^{\epsilon}(0)\,]_{q^{-2}}=0,\\ \label{f:comm4}%3.9
%& p^{b_{0i}}[\,x_0^{\epsilon}(-\epsilon),\,x_i^{\pm}(0)\,]_{q^{\pm (\alpha_0, \alpha_i)}}+
%[\,x_0^{\epsilon}(-\epsilon),\,x_i^{\pm}(0)\,]_{q^{\pm (\alpha_i, \alpha_j)}}=0, \\
&[\,x_i^{+}(0),\,x_j^{-}(0)\,]=\delta_{ij}\frac{K_i-K_i^{-1}}{q-q^{-1}},\qquad [\,x_0^{+}(-1),\,x_0^{-}(1)\,]=\frac{\gamma^{-1}K_0-\gamma K_0^{-1}}{q-q^{-1}},\\
 &[\,x_i^{-\epsilon}(0),\,x_0^{\epsilon}(-\epsilon)\,]=0,
 %\qquad
% [\,x_0^{+}(-1),\,x_i^{-}(0)\,]=0,
\quad \hbox{for} \quad i\neq 0,\\
%\end{align}
%\begin{align}\label{f:comm5a}
&[x_i^{\pm}(0), x_j^{\pm}(0)]=0, \qquad [x_0^{\epsilon}(-\epsilon), x_k^{\epsilon}(0)]=0, \quad\mbox{for}\ \  a_{ij}=a_{0k}=0,\\
\label{f:comm5}
%&[\,x_0^{\epsilon}(-\epsilon),\,[x_0^{\epsilon}(-\epsilon),\,[\,x_0^{\epsilon}(-\epsilon),\,x_0^{-\epsilon}(0)\,]_{1}]_{q^{2}}]_{q^4}=0,\\
%&[\,x_0^{\epsilon}(0),\,[x_0^{\epsilon}(0),\,[\,x_0^{\epsilon}(0),\,x_0^{-\epsilon}(\epsilon)\,]_{1}]_{q^{-2}}]_{q^{-4}}=0,\\
&x_i^{\pm}(0)x_i^{\pm}(0)x_j^{\pm}(0)-[2]x_i^{\pm}(0)x_j^{\pm}(0)x_i^{\pm}(0)
+x_j^{\pm}(0)x_i^{\pm}(0)x_i^{\pm}(0)=0, \qquad\mbox{$a_{ij}=-1$},\\ \label{f:comm6}
&x_j^{\epsilon}(0)x_j^{\epsilon}(0)x_0^{\epsilon}(-\epsilon)-[2]x_j^{\epsilon}(0)x_0^{\epsilon}(-\epsilon)x_j^{\epsilon}(0)+
x_0^{\epsilon}(-\epsilon)x_j^{\epsilon}(0)x_j^{\epsilon}(0)=0, \qquad\mbox{$a_{0j}=-1$},\\ \label{f:comm7}
&x_0^{\epsilon}(-\epsilon)x_0^{\epsilon}(0)x_j^{\epsilon}(0)-[2]x_0^{\epsilon}(-\epsilon)x_j^{\epsilon}(0)x_0^{\epsilon}(0)+
x_j^{\epsilon}(0)x_0^{\epsilon}(-\epsilon)x_0^{\epsilon}(0) \\ \nonumber
&+x_0^{\epsilon}(0)x_0^{\epsilon}(-\epsilon)x_j^{\epsilon}(0)-[2]x_0^{\epsilon}(0)x_j^{\epsilon}(0)x_0^{\epsilon}(-\epsilon)+
x_j^{\epsilon}(0)x_0^{\epsilon}(0)x_0^{\epsilon}(-\epsilon)=0, \qquad\mbox{$a_{0j}=-1$},
\end{align}
where $\epsilon=\pm $ or $\pm 1$, $A=(a_{ij})$ the Cartan matrix of $A_n^{(1)}$.
%and $b_{ij}=\delta_{i, j-1}-\delta_{i, j+1}, i, j\in I=\mathbb Z_{n+1}=\{0, 1, \ldots, n\}$.
%are the entries of the antisymmetric matrix associated to $A_n^{(1)}$.
\end{defi}

%In fact, the imaginary quantum root vectors $a_i(l)$ in $U_q(sl_{n+1},tor)$ can be generated by the above finite generators.
It is clear that $\mathcal U_0$ is finitely generated with finitely many relations, and has
fewer generators and simpler relations than Drinfeld's original form. %Through several steps
We will show that $\mathcal U_0$
is isomorphic to a quotient algebra of $U_q(sl_{n+1},tor)$. Observe that $\mathcal U_0$ has a Chevelley anti-involution $\iota: x_i^{\pm}(k)\mapsto x_i^{\mp}(-k)$,
$K_i\mapsto K_i^{-1}$,
$q^{d_1}\mapsto q^{-d_1}$, $q^{d_2}\mapsto q^{d_2}$, $\gamma\mapsto \gamma^{-1}$ and $q\mapsto q^{-1}$ over the complex field.

%We now start to show that the algebra $\mathcal U_0$ is isomorphic to the quantum toroidal algebra
%using .

\begin{theo}\label{newpre}\, There exists an epimorphism $\pi: {\mathcal U}_0 \longrightarrow U_q(sl_{n+1},tor)$ such that $\pi$ is
identity on the set of
%$\pi(x)=x$, here $x$ in the set of
generators of ${\mathcal U}_0$ in Definition \ref{4.2}. That is, ${\mathcal U}_0/ker\pi\cong U_q(sl_{n+1},tor)$.
\end{theo}

 To prove Theorem \ref{newpre}, we need to check $\pi$
 can be extended from the generating set to the whole algebra and preserves all the relations.
 %.is a surjective algebra homomorphism.
 To this end,
 we introduce the following elements:
\begin{gather}\label{a1}
a_0(1)=K_0^{-1}\gamma^{1/2}\,\bigl[\,x_0^+(0),\,x_0^-(1)\,\bigr]\in
{\mathcal U}_0,\\\label{a2}
a_0(-1)=K_0{\gamma}^{-1/2}\,\bigl[\,x_0^+(-1),\,x_0^-(0)\,
\bigr]\in
{\mathcal U}_0,
\end{gather}
and use them to generate higher degree elements in a spiral way.
%based on Lemma \ref{L:iso}.
As $\iota(a_0(\pm 1))=a_0(\mp 1)$ we only check half of the relations.
The following relations involving with  $a_0(\pm 1)$ %$\epsilon=\pm 1$
are clear.

\begin{prop}\label{p1} Using the above notations, the following relations are
compatible with the defining relations of the quantum toroidal algebra $U_q(sl_{n+1},tor)$ ($\epsilon=\pm$ or $\pm 1$):
\begin{align}
\label{b1}
&K_ia_0(\epsilon)K_i^{-1}=a_0(\epsilon),\qquad q^{d_1}a_0(\epsilon)q^{-d_1}=q^{\epsilon} a_0(\epsilon),
\qquad q^{d_2}a_0(\epsilon)q^{-d_2}=a_0(\epsilon),\\\label{b2}
&\bigl[\,a_0(\epsilon),\,x_0^{\epsilon}(-\epsilon)\,\bigr]=\epsilon [2]\gamma^{-\frac{\epsilon}{2}} x_0^{\epsilon}(0),\quad \bigl[\,a_0(-\epsilon),\,x_0^{\epsilon}(0)\,\bigr]=\epsilon [2]\gamma^{-\frac{\epsilon}{2}}x_0^{\epsilon}(-\epsilon),\\
\label{b3}
&\bigl[\,a_0(1), a_0(-1)\,\bigr]=[2]\frac{\gamma-\gamma^{-1}}{q-q^{-1}}.
\end{align}
\end{prop}
Here the compatibility means that the map $\pi:$ $a_0(\epsilon)\mapsto a_0(\epsilon)$, $x_0^{\epsilon}(-\epsilon)\mapsto
x_0^{\epsilon}(-\epsilon)$ is a homomorphism. This meaning will be used repeatedly in the following.
\begin{proof}\,
Using the Chevalley anti-involution, we only need to show the relations for $\epsilon=1$ (viewed as $+$ in superscript).
\eqref{b1}-\eqref{b3} follow by direction computation.
\end{proof}

Furthermore, for $\epsilon=\pm$ or $\pm1$ we have
\begin{gather}\label{a3}
x_0^{\epsilon}(\epsilon)=[2]^{-1}\gamma^{\frac{\epsilon}{2}}\,\bigl[\,a_0(\epsilon),\,x_0^{\epsilon}(0)\,\bigr]\in
{\mathcal U}_0.
\end{gather}

In fact, using the above notations and \eqref{a3}, we have the following lemma.
\begin{lemm}\label{p2} One has that for $\epsilon=\pm 1$
\begin{align}\label{b4}
&\bigl[\,x_0^+(\epsilon),\,x_0^-(0)\,\bigr]=\gamma \bigl[\,x_0^+(0),\,x_0^-(\epsilon)\,\bigr],
%\label{b5}
%&\bigl[\,x_0^+(0),\,x_0^-(-1)\,\bigr]=\gamma^{-1} \bigl[\,x_0^+(-1),\,x_0^-(0)\,\bigr],
\end{align}
which are equivalent to
\begin{align}\label{b6}
&a_0(\pm 1)=\pm K_0^{\mp 1}\gamma^{\pm\frac{1}{2}}\bigl[\,x_0^{\pm}(0),\,\,x_0^{\mp}(\pm 1)\bigr]=\pm K_0^{\mp 1}\gamma^{\mp\frac{1}{2}} \bigl[\,x_0^{\pm}(\pm 1), x_0^{\mp}(0)\bigr].
%&a_0(1)=K_0^{-1}\gamma^{\frac{1}{2}}\bigl[\,x_0^+(0),\,x_0^-(1)\,\bigr]=K_0^{-1}\gamma^{-\frac{1}{2}} \bigl[\,x_0^+(1),\,x_0^-(0)\,\bigr]\\
%\label{b7}
%&a_0(-1)=K_0\gamma^{-\frac{1}{2}}\bigl[\,x_0^+(-1),\,x_0^-(0)\,\bigr]=K_0\gamma^{\frac{1}{2}} \bigl[\,x_0^+(0),\,x_0^-(-1)\,\bigr]
\end{align}
\end{lemm}

%\begin{proof}\, \eqref{b4}-\eqref{b5} follow from the notation \eqref{a3}.
%\begin{equation*}
%\begin{split}
%&\bigl[\,x_0^+(1),\,x_0^-(0)\,\bigr]
%=[2]^{-1}\gamma^{\frac{1}{2}}\,\bigl[\,[a_0(1),\,x_0^+(0)],\,x_0^-(0)\,\bigr]\\
%&=[2]^{-1}\gamma^{\frac{1}{2}}\,\Bigl(\bigl[\,[a_0(1),\,x_0^-(0)],\,x_0^+(0)\,\bigr]
%+\bigl[\,a_0(1),\,[x_0^+(0),\,x_0^-(0)]\,\bigr]\Bigr)\\
%&=\gamma \bigl[\,x_0^+(0),\,x_0^-(1)\,\bigr]
%\end{split}
%\end{equation*}
%
%Similarly, it holds by using the notation \eqref{a3},
%\begin{equation*}
%\begin{split}
%&\bigl[\,x_0^+(0),\,x_0^-(-1)\,\bigr]
%=-[2]^{-1}\gamma^{-\frac{1}{2}}\,\bigl[\,x_0^+(0),\,[a_0(-1),\,x_0^-(0)]\,\bigr]\\
%&=-[2]^{-1}\gamma^{-\frac{1}{2}}\,\Bigl(\bigl[\,[x_0^+(0),\,a_0(-1)],\,x_0^-(0)]\,\bigr]
%+\bigl[\,a_0(-1),\,[x_0^+(0),\,x_0^-(0)]\,\bigr]\Bigr)\\
%&=\gamma^{-1} \bigl[\,x_0^+(-1),\,x_0^-(0)\,\bigr]
%\end{split}
%\end{equation*}
%which is compatible with  the defining relation \eqref{e:comm4}.
%\end{proof}

\begin{prop} From the above constructions, we have the following relations, which are compatible to the defining relations of the quantum toroidal algebra $U_q(sl_{n+1},tor)$.
\begin{align}\label{b8}
&K_ix_0^{\epsilon}(\epsilon)K_i^{-1}=q^{a_{i0}}x_0^{\epsilon}(\epsilon), \quad q^{d_1}x_0^{\epsilon}(\epsilon)q^{-d_1}=q^{\epsilon} x_0^{\epsilon}(\epsilon),\quad q^{d_2}x_0^{\epsilon}(\epsilon)q^{-d_2}=q x_0^{\epsilon}(\epsilon),\\\label{b9}
&[\,x_0^{\pm}(1),\,x_0^{\pm}(-1)\,]_{q^{\pm2}}+[\,x_0^{\pm}(0),\,x_0^{\pm}(0)\,]_{q^{\pm2}}=0,\\\label{b10}
&\bigl[\,a_0(-\epsilon),\,x_0^{\epsilon}(\epsilon)\,\bigr]=\epsilon [2]\gamma^{-\frac{\epsilon}{2}} x_0^{\epsilon}(0),\\ %\bigl[\,a_0(\epsilon),\,x_0^{\epsilon}(0)\,\bigr]=\epsilon [2]\gamma^{-\frac{\epsilon}{2}}x_0^{\epsilon}(\epsilon),
\label{b11}
&[\,x_0^+(1),\,x_0^-(-1)\,]=\frac{\gamma K_0-\gamma^{-1} K_0^{-1}}{q-q^{-1}}.
\end{align}
\end{prop}
\begin{proof}\, By the Chevalley anti-involution, it is enough to show the relations for $\epsilon=1$.
%In deed, \eqref{b8} can be checked directly.
% we have that,
%\begin{equation*}
%\begin{split}
%&K_ix_0^+(1)K_i^{-1}=K_i[2]^{-1}\gamma^{\frac{1}{2}}\,\bigl[\,a_0(1),\,x_0^+(0)\,\bigr]K_i^{-1}
%=q^{a_{i0}}x_0^+(1)
%\end{split}
%\end{equation*}
%and
%\begin{equation*}
%\begin{split}
%&q^{d_1}x_0^+(1)q^{-d_1}=q^{d_1}[2]^{-1}\gamma^{\frac{1}{2}}\,\bigl[\,a_0(1),\,x_0^+(0)\,\bigr]q^{-d_1}
%=q [2]^{-1}\gamma^{\frac{1}{2}}\,\bigl[\,a_0(1),\,x_0^+(0)\,\bigr]=q x_0^+(1)
%\end{split}
%\end{equation*}
%which are compatible with  the defining relation \eqref{e:comm00} and \eqref{e:comm0a} respectively.
To check \eqref{b9}, note that $[\,x_0^{+}(0),\,x_0^{+}(-1)\,]_{q^{2}}=0$ by \eqref{f:comm3}, then
\begin{equation*}
\begin{split}
%0&=[\,a_0(1),\,A_1\,]\\
0&=\Big[\,a_0(1),\,[\,x_0^{+}(0),\,x_0^{+}(-1)\,]_{q^{2}}\,\Big]\\
&=\Bigl(\Big[\,[a_0(1),\,x_0^{+}(0)],\,x_0^{+}(-1)\,\Big]_{q^{2}}+
\Big[\,x_0^{+}(0),\,[a_0(1),\,x_0^{+}(-1)\,]\,\Big]_{q^{2}}\Big)\\
&=[2]\gamma^{-\frac{1}{2}}\Big([\,x_0^{+}(1),\,x_0^{+}(-1)\,]_{q^{2}}
+[\,x_0^{+}(0),\,x_0^{+}(0)\,]_{q^{2}}\Big).
\end{split}
\end{equation*}
It means that $[\,x_0^{+}(1),\,x_0^{+}(-1)\,]_{q^{2}}
+[\,x_0^{+}(0),\,x_0^{+}(0)\,]_{q^{2}}=0$, which is compatible with the defining relation \eqref{e:Dr7}. %\eqref{e:comm3}.
Similarly \eqref{b10} follows from  \eqref{a2} and \eqref{f:comm4} via \eqref{b6} and \eqref{b9}.
%\begin{equation*}
%\begin{split}
%&[\,a_0(-1),\,x_0^+(1)\,]\\
%&=\Big[\,K_0\gamma^{-\frac{1}{2}}\,\bigl[\,x_0^+(-1),\,x_0^-(0)\,\bigr],\,x_0^+(1)\,\Big]\\
%&=K_0\gamma^{-\frac{1}{2}}\,\Big[\,\bigl[\,x_0^+(-1),\,x_0^-(0)\,\bigr],\,x_0^+(1)\,\Big]_{q^{-2}}\\
%&=K_0\gamma^{-\frac{1}{2}}\,\Bigl(\Big[\,x_0^+(-1),\,\bigl[\,x_0^-(0),\,x_0^+(1)\,\bigr]\,\Bigr]_{q^{-2}}
%+\Big[\,\bigl[\,x_0^+(-1),\,x_0^+(1)\,\bigr]_{q^{-2}},\,x_0^-(0)\,\Big]\Bigr)\\
%&=K_0\gamma^{-\frac{1}{2}}\,\Bigl(q^{-2}[2]K_0x_0^+(0)-q^{-2}[2]K_0x_0^+(0)+[2]K_0^{-1}x_0^+(0)\Big)\\
%&=[2]\gamma^{-1/2}x_0^+(0),
%\end{split}
%\end{equation*}
%where we have used  \eqref{b7} and \eqref{b9}.
%similarly, one has that,
%\begin{equation*}
%\begin{split}
%&[\,a_0(1),\,x_0^+(0)\,]\\&=\Big[\,K_0^{-1}\gamma^{1/2}\,\bigl[\,x_0^+(0),\,x_0^-(1)\,\bigr],\,x_0^+(0)\,\Big]\\
%&=K_0^{-1}\gamma^{1/2}\,\Bigl(\Big[\,x_0^+(0),\,\bigl[\,x_0^-(1),\,x_0^+(0)\,\bigr]\,\Bigr]_{q^{2}}
%+\Big[\,\bigl[\,x_0^+(0),\,x_0^+(0)\,\bigr]_{q^{2}},\,x_0^-(1)\,\Big]\Bigr)\\
%&=\Bigl(-q^{-2}(x_0^+(0))\\
%&=-[2]\gamma^{1/2}x_0^-(1)
%\end{split}
%\end{equation*}
\eqref{b11} holds from direct calculation.
%\begin{equation*}
%\begin{split}
%&[\,x_0^+(1),\,x_0^-(-1)\,]\\
%%&=\Big[\,K_0^{-1}\gamma^{1/2}\,\bigl[\,x_0^+(0),\,x_0^-(1)\,\bigr],\,x_1^+(0)\,\Big]\\
%&=[2]^{-2}\,\Big[\,\bigl[\,a_0(1),\,x_0^+(0)\,\bigr],\,\bigl[\,a_0(-1),\,x_0^-(0)\,\bigr]\,\Big]\\
%%&=[2]^{-2}\,\Bigl(\Big[\,\bigl[\,a_0(1),\,[\,x_0^+(0),\,a_0(-1)\,]\,\bigr],\,x_0^-(0)\,\Big]
%%+\Big[\,a_0(-1),\,\bigl[\,[\,a_0(1),\,x_0^-(0)\,],\,x_0^+(0)\,\bigr]\,\Big]\Bigr)\\
%&=-[2]^{-1}\,\Bigl(\gamma^{-\frac{1}{2}}\Big[\,\bigl[\,a_0(1),\,x_0^+(-1)\,\bigr],\,x_0^-(0)\,\Big]
%+\gamma^{\frac{1}{2}}\Big[\,a_0(-1),\,\bigl[\,x_0^-(1),\,x_0^+(0)\,\bigr]\,\Big]\Bigr)\\
%%&=\Bigl(\gamma^{-1}\frac{K_0-K_0^{-1}}{q-q^{-1}}+K_0\frac{\gamma-\gamma^{-1}}{q-q^{-1}}\Bigr)\\
%&=\frac{\gamma K_0-\gamma^{-1} K_0^{-1}}{q-q^{-1}}
%\end{split}
%\end{equation*}
%which is compatible with  the defining relation \eqref{e:comm4}.
\end{proof}
Now we construct all degree-$k$ elements $x_0^{\pm}(k),\,x_0^{\pm}(-k),\,a_0(\pm k)$ involving with index $i=0$ by induction on the degree as follows.  For $\epsilon=\pm$ or $\pm1$, we denote that,
\begin{gather}\label{a4}
x_{0}^{\pm}(\epsilon k)=\pm [2]^{-1}\gamma^{\pm\frac{1}{2}}\,\bigl[\,a_0(\epsilon),\,x_{0}^{\pm}(\epsilon (k-1))\,\bigr]\in
{\mathcal U}_0,\\\label{a5}
%x_{0}^{\pm}(-k-1)=\pm [2]^{-1}\gamma^{\pm\frac{1}{2}}\,\bigl[\,a_0(-1),\,x_{i}^{\pm}(-k)\,\bigr]\in
%{\mathcal U}_0,\\\label{a6}
\phi_{i}(k)=(q-q^{-1})\gamma^{\frac{2-k}{2}}\,\bigl[\,x_{0}^+(k-1),\,x_{0}^{-}(1)\,\bigr]\in
{\mathcal U}_0,\\\label{a6}
{\varphi}_{0}(-k)=-(q-q^{-1})\gamma^{\frac{k-2}{2}}\,\bigl[\,x_{0}^+(-1),\,x_{0}^{-}(-k+1)\,\bigr]\in
{\mathcal U}_0,
\end{gather}
where $a_0(\pm k)$ are defined by $\phi_0(k)$ and ${\varphi}_0(-k)\, (k\geq 0)$ %such that $\phi_i(0)=K_i$ and  ${\varphi}_i(0)=K_i^{-1}$
as follows:
\begin{gather*}\sum\limits_{m=0}^{\infty}\phi_0(\pm r) z^{\mp r}=K_0^{\pm 1} \exp \Big(\pm
(q-q^{-1})\sum\limits_{r=1}^{\infty}
 a_0(\pm r)z^{\mp r}\Big).
%\sum\limits_{r=0}^{\infty}{\varphi}_0(-r) z^{r}=K_0^{-1}\exp
%\Big({-}(q{-}q^{-1})
%\sum\limits_{r=1}^{\infty}a_0(-r)z^{r}\Big).
\end{gather*}

A partition $\lambda$ of $k$, denoted $\lambda\vdash k$, is a decreasing sequence of positive integers $\lambda_1\geq\lambda_2\geq\cdots\geq \lambda_l>0 $ such that $\lambda_1+\lambda_2+\cdots+\lambda_l=k$, where $l(\lambda)=l$ is called the number of parts.
$\lambda$ can also be written as $(1^{m_1}2^{m_2}\cdots)$ with multiplicity of $i$ being $m_i$.
Then we obtain the following formulas between $a_0(\pm k)$ and  $\phi_0(k)$ or ${\varphi}_0(-k)$:
 \begin{align}
 \label{a7}
 \phi_0(\pm k)&=K_0^{\pm 1}\sum\limits_{\lambda\vdash k}\frac{(q^{\pm 1}-q^{\mp 1})^{l(\lambda)}}{m_{\lambda}!}a_0(\pm\lambda),
 % &\phi_0(k)=K_0\sum\limits_{t=1}^{k}\frac{(q-q^{-1})^t}{t!}\sum \limits_{\substack{ 1\leqslant i_1,\cdots,i_t\leqslant k-t+1,\\
 % i_1+i_2+\cdots+i_t=k}} a_0(i_1)\cdots a_0(i_t),
  %\\ \label{a8}
% \phi_0(-k)&=K_0^{-1}\sum\limits_{\lambda\vdash k}\frac{(q^{-1}-q)^{l(\lambda)}}{m_{\lambda}!}a_0(-\lambda),
%&{\varphi}_0(-k)=K_0^{-1}\sum\limits_{t=1}^{k}\frac{(q^{-1}-q)^t}{t!}\sum\limits_{\substack {1\leqslant i_1,\cdots,i_t\leqslant k-t+1,\\
%  i_1+i_2+\cdots+i_t=k}} a_0(-i_1)\cdots a_0(-i_t).
\end{align}
where $m_{\lambda}!=\prod_{i\geq 1}m_i!$
and $a_0(\pm\lambda)=a_0(\pm\lambda_1)a_0(\pm\lambda_2)\cdots.$
It is clear that $\iota( \phi_0(k))= \varphi_0(-k)$.

The following result can be directly checked, but we will give a general proof in Prop. \ref{p11}.
\begin{lemm} For $\epsilon=\pm$ or $\pm1$, one has that
\begin{align}\label{b12}
&[\,a_0(2\epsilon),\,x_0^{\pm}(-\epsilon)\,]=\frac{[4]}{2}\gamma^{\mp1}x_0^{\pm}(\epsilon),\\
\label{b13}
&[\,a_0(2\epsilon),\,x_0^{\pm}(0)\,]=\frac{[4]}{2}\gamma^{\mp1}x_0^{\pm}(2\epsilon),\\\label{b14}
&[\,a_0(2),\,a_0(-1)]=0,\\\label{b15}
&[\,a_0(2),\,a_0(-2)]=\frac{[4]}{2}\frac{\gamma^2-\gamma^{-2}}{q-q^{-1}}.
\end{align}
\end{lemm}

Using these relations, we have the following result.
\begin{prop}For $\epsilon=\pm$ or $\pm1$, it holds that,
\begin{align}\label{b16}
&[\,x_0^{\pm}(2),\,x_0^{\pm}(-1)\,]_{q^{\pm2}}
+[\,x_0^{\pm}(0),\,x_0^{\pm}(1)\,]_{q^{\pm2}}=0,\\\label{b17}
&[\,x_0^{\epsilon}(\epsilon),\,x_0^{\epsilon}(0)\,]_{q^{2}}=0,
\end{align}
which are compatible with the defining relation \eqref{e:Dr7}. %\eqref{e:comm3}.
\end{prop}
\begin{proof} Using the Chevalley anti-involution $\iota$, it is enough to show the case of $\epsilon=+$. To check \eqref{b16}, note that $[\,x_0^{+}(0),\,x_0^{+}(-1)\,]_{q^{2}}=0$.
 Thus it follows from \eqref{b12} and \eqref{b13} that
\begin{equation*}
\begin{split}
0=[\,a_0(2), [\,x_0^{+}(0),\,x_0^{+}(-1)\,]_{q^{2}}\,]
=\frac{[4]}{2}\gamma^{-1}\Big([\,x_0^{+}(2),\,x_0^{+}(-1)\,]_{q^{2}}
+[\,x_0^{+}(0),\,x_0^{+}(1)\,]_{q^{2}}\Big),
\end{split}
\end{equation*}
which implies that $[\,x_0^{+}(2),\,x_0^{+}(-1)\,]_{q^{2}}
+[\,x_0^{+}(0),\,x_0^{+}(1)\,]_{q^{2}}=0$.

%For \eqref{b17}, we denote that $A_2=[\,x_0^{+}(1),\,x_0^{+}(-1)\,]_{q^{2}}
%+[\,x_0^{+}(0),\,x_0^{+}(0)\,]_{q^{2}}=0$ by \eqref{b9}.
Similarly by \eqref{b9},
\begin{equation*}
\begin{split}
0&=[\,a_0(1), [\,x_0^{+}(1),\,x_0^{+}(-1)\,]_{q^{2}}
+[\,x_0^{+}(0),\,x_0^{+}(0)\,]_{q^{2}}\,]\\
%&=\Big[\,a_0(1),\,[\,x_0^{+}(1),\,x_0^{+}(-1)\,]_{q^{2}}
%+[\,x_0^{+}(0),\,x_0^{+}(0)\,]_{q^{2}}\,\Big]\\
&=[2]\gamma^{-\frac{1}{2}}\Bigl([x_0^{+}(2),x_0^{+}(-1)]_{q^{2}}+
[x_0^{+}(1), x_0^{+}(0)]_{q^{2}}+[x_0^{+}(1), x_0^{+}(0)]_{q^{2}}
+[x_0^{+}(0), x_0^{+}(1)]_{q^{2}}\Bigr)\\
&=2[2]\gamma^{-\frac{1}{2}}[\,x_0^{+}(1),\,x_0^{+}(0)\,]_{q^{2}},\\
\end{split}
\end{equation*}
where we have used \eqref{b16}.  It yields that $[\,x_0^{+}(1),\,x_0^{+}(0)\,]_{q^{2}}=0$.
\end{proof}

So far we have shown that $x_i^{\pm}(0),\, x_0^{\pm}(1),\,x_0^{\pm}(-1),\,a_0(\pm 1)$  ($i\in I$) in the algebra ${\mathcal U}_0$ satisfy relations \eqref{e:comm0} to \eqref{e:comm0c} and \eqref{e:Dr3} to \eqref{e:Dr11}. Then we can use induction on the degree $k$ to show that
$x_0^{\pm}(k), a_0(\pm k)$ are all in the algebra and satisfy the Drinfeld relations.
To this end, suppose all degree-k ($k \leqslant n$) elements  $x_0^{\pm}(\epsilon k),\,a_0(\epsilon k)$,
 $x_i^{\pm}(0)$ for $i\in I$  are in the algebra ${\mathcal U}_0$, and satisfy relations \eqref{e:comm0} to \eqref{e:comm0c} and \eqref{e:Dr3} to \eqref{e:Dr11}.
Next we need to verify that all the elements $x_0^{\pm}(\epsilon(n+1)),\,a_0(\epsilon(n+1))$  for $i\in I$  satisfy the relations \eqref{e:comm0} to \eqref{e:comm0c} and \eqref{e:Dr3} to \eqref{e:Dr11}. We need a couple of lemmas for this.

%The following is clear.
\begin{lemm} One has that
\begin{align}\label{b18}
[\,a_0(1),\,\phi_0(n)\,]=0,
%\label{b19}
\quad [\,a_0(-1),\,{\varphi}_0(-n)\,]=0.
\end{align}
\end{lemm}
%\begin{proof}\,\eqref{b19} follows from \eqref{b18} by the Chevally anti-involution $\iota$. For \eqref{b18}, it holds from inductive hypothesis,
%\begin{equation*}
%\begin{split}
%&\bigl[\,a_0(1),\,\phi_0(n)\,\bigr]\\
%%&=(q-q^{-1})\gamma^{\frac{2-n}{2}}\,\bigl[\,[a_0(1),\,[x_0^+(n-1),\,x_0^-(1)]\,\bigr]\\
%&=(q-q^{-1})\gamma^{\frac{2-n}{2}}\,\Big(\bigl[\,[a_0(1),\,x_0^+(n-1)],\,x_0^-(1)\,\bigr]+
%\bigl[\,x_0^+(n-1),\,[a_0(1),\,x_0^-(1)]\,\bigr]\Big)\\
%&=[2]\phi_0(n+1)-[2]\phi_0(n+1)\\
%&=0.
%\end{split}
%\end{equation*}
%\end{proof}

\begin{lemm} It is clear that,
\begin{align}\label{b20}
&\frac{1}{(q-q^{-1})}\phi_0(n+1)=\gamma^{-\frac{n+1}{2}}\bigl[\,x_0^+(n+1),\,x_0^-(0)\,\bigr]
=\gamma^{\frac{1-n}{2}} \bigl[\,x_0^+(n),\,x_0^-(1)\,\bigr]\\ \nonumber
&\hskip1.8cm=\cdots=\gamma^{\frac{n-1}{2}} \bigl[\,x_0^+(1),\,x_0^-(n)\,\bigr]=\gamma^{\frac{n+1}{2}} \bigl[\,x_0^+(0),\,x_0^-(n+1)\,\bigr],\\
\label{b21}
&\frac{-1}{(q-q^{-1})}{\varphi}_0(-n-1)=\gamma^{-\frac{n+1}{2}}\bigl[\,x_0^+(-n-1),\,x_0^-(0)\,\bigr]=\gamma^{\frac{1-n}{2}} \bigl[\,x_0^+(-n),\,x_0^-(-1)\,\bigr]\\ \nonumber
&\hskip1.8cm=\cdots =\gamma^{\frac{n-1}{2}}\bigl[\,x_0^+(-1),\,x_0^-(-n)\,\bigr]
=\gamma^{\frac{n+1}{2}}\bigl[\,x_0^+(0),\,x_0^-(-n-1)\,\bigr].
\end{align}

\end{lemm}

\begin{proof}\, In fact, \eqref{b21} can be obtained from \eqref{b20} by the Chevally anti-involution $\iota$. \eqref{b20} follows from inductive hypothesis.
%\begin{equation*}
%\begin{split}
%&\bigl[\,x_0^+(n+1),\,x_0^-(0)\,\bigr]
%=[2]^{-1}\gamma^{\frac{1}{2}}\,\bigl[\,[a_0(1),\,x_0^+(n)],\,x_0^-(0)\,\bigr]\\
%=&[2]^{-1}\gamma^{\frac{1}{2}}\,\Bigl(\bigl[\,[a_0(1),\,x_0^-(0)],\,x_0^+(n)\,\bigr]
%+\bigl[\,a_0(1),\,[x_0^+(n),\,x_0^-(0)]\,\bigr]\Bigr)\\
%=&\gamma \bigl[\,x_0^+(n),\,x_0^-(1)\bigr]=\cdots
%\end{split}
%\end{equation*}
%which is compatible with  the defining relation \eqref{e:comm4}.
\end{proof}

We turn to the following proposition.

\begin{prop}\label{p10} From the above constructions, we have the following relations for $\epsilon=\pm$ or $\pm1$, $-n+1\leqslant l_1\leqslant n-1$ and $-n\leqslant l_2\leqslant n$,
\begin{align}\label{b22}
&K_ix_{0}^{\pm}(\epsilon(n+1))K_i^{-1}=q^{a_{i0}}x_{0}^{\pm}(\epsilon(n+1)),\\\label{b23}
&q^{d_1}x_{0}^{\pm}(\epsilon(n+1))q^{-d_1}=q^{\epsilon(n+1)} x_{0}^{\pm}(\epsilon(n+1)),\\\label{b24}
&q^{d_2}x_{0}^{\pm}(\epsilon(n+1))q^{-d_2}=q x_{0}^{\pm}(\epsilon(n+1)),\\\label{b25}
&[\,x_{0}^{\pm}(n+1),\,x_{0}^{\pm}(n-1)\,]_{q^{\pm2}}+[\,x_{0}^{\pm}(n),\,x_{0}^{\pm}(n)\,]_{q^{\pm2}}=0,\\\label{b26}
&[\,x_{0}^{\pm}(n+2),\,x_{0}^{\pm}(n-1)\,]_{q^{\pm2}}+[\,x_{0}^{\pm}(n),\,x_{0}^{\pm}(n+1)\,]_{q^{\pm2}}=0,\\
\label{b27}
&[\,x_{0}^{\pm}(n+1),\,x_{0}^{\pm}(n)\,]_{q^{\pm2}}=0,\\
\label{b28}
&[\,x_0^{\pm}(n+1),\,x_0^{\pm}(l_1)\,]_{q^{\pm2}}+[\,x_{0}^{\pm}(l_1+1),\,x_{0}^{\pm}(n)\,]_{q^{\pm2}}=0,\\\label{b29}
&[\,x_{0}^+(n+1),\,x_{0}^-(l_2)\,]=\frac{\gamma^{\frac{n+1-l}{2}}\phi_0(n+l_2+1)}{q-q^{-1}},\\\label{b30}
&[\,x_0^+(n+1),\,x_0^-(-n-1)\,]=\frac{\gamma^{n+1} K_0-\gamma^{-n-1} K_0^{-1}}{q-q^{-1}}.
\end{align}
\end{prop}

\begin{proof} We only check  the case ``+'', the case ``-'' can be obtained similarly. \eqref{b22}, \eqref{b23} and \eqref{b24} hold directly from the constructions.
To check \eqref{b25}, we have that by inductive hypothesis, % $$A_3=[\,x_0^{+}(n),\,x_0^{+}(n-1)\,]_{q^{2}}=0.$$
\begin{equation*}
\begin{split}
0&=[\,a_0(1), [\,x_0^{+}(n),\,x_0^{+}(n-1)\,]_{q^{2}}\,]\\
%&=\Big[\,a_0(1),\,[\,x_0^{+}(n),\,x_0^{+}(n-1)\,]_{q^{2}}\,\Big]\\
%&=\Bigl(\Big[\,[a_0(1),\,x_0^{+}(n)],\,x_0^{+}(n-1)\,\Big]_{q^{2}}+
%\Big[\,x_0^{+}(n),\,[a_0(1),\,x_0^{+}(n-1)\,]\,\Big]_{q^{2}}\Big)\\
&=[2]\gamma^{-\frac{1}{2}}\Big([\,x_0^{+}(n+1),\,x_0^{+}(n-1)\,]_{q^{2}}
+[\,x_0^{+}(n),\,x_0^{+}(n)\,]_{q^{2}}\Big),
\end{split}
\end{equation*}
which implies that $[\,x_0^{+}(n+1),\,x_0^{+}(n-1)\,]_{q^{2}}
+[\,x_0^{+}(n),\,x_0^{+}(n)\,]_{q^{2}}=0$.

Applying $a_0(2)$ to $[\,x_0^{+}(n),\,x_0^{+}(n-1)\,]_{q^{2}}=0$, we also get
%\begin{equation*}
%\begin{split}
%0&=[\,a_0(2),\,A_3\,]\\
%%&=\Big[\,a_0(2),\,[\,x_0^{+}(n),\,x_0^{+}(n-1)\,]_{q^{2}}\,\Big]\\
%&=\Bigl(\Big[\,[a_0(2),\,x_0^{+}(n)],\,x_0^{+}(n-1)\,\Big]_{q^{2}}+
%\Big[\,x_0^{+}(n),\,[a_0(2),\,x_0^{+}(n-1)\,]\,\Big]_{q^{2}}\Big)\\
%&=\frac{[4]}{2}\gamma^{-1}\Big([\,x_0^{+}(n+2),\,x_0^{+}(n-1)\,]_{q^{2}}
%+[\,x_0^{+}(n),\,x_0^{+}(n+1)\,]_{q^{2}}\Big),
%\end{split}
%\end{equation*}
$$[\,x_0^{+}(n+2),\,x_0^{+}(n-1)\,]_{q^{2}}
+[\,x_0^{+}(n),\,x_0^{+}(n+1)\,]_{q^{2}}=0.$$

Repeatedly applying $a_0(1)$ to brackets like $[\,x_0^{+}(m),\,x_0^{+}(n)\,]_{q^{2}}$, we can prove
\eqref{b27}-\eqref{b28}.

%For \eqref{b27}, denote that $A_4=[\,x_0^{+}(n+1),\,x_0^{+}(n-1)\,]_{q^{2}}
%+[\,x_0^{+}(n),\,x_0^{+}(n)\,]_{q^{2}}=0$ by \eqref{b25}. Therefore, it is easy to see that,
%\begin{equation*}
%\begin{split}
%0&=[\,a_0(1),\,A_4\,]\\
%%&=\Big[\,a_0(1),\,[\,x_0^{+}(n+1),\,x_0^{+}(n-1)\,]_{q^{2}}
%%+[\,x_0^{+}(n),\,x_0^{+}(n)\,]_{q^{2}}\,\Big]\\
%&=[2]\gamma^{-\frac{1}{2}}\Bigl([x_0^{+}(n+2),x_0^{+}(n-1)]_{q^{2}}+
%[x_0^{+}(n+1), x_0^{+}(n)]_{q^{2}}\\
%&\hskip1.85cm+[x_0^{+}(n+1), x_0^{+}(n)]_{q^{2}}
%+[x_0^{+}(n), x_0^{+}(n+1)]_{q^{2}}\Bigr)\\
%&=2[2]\gamma^{-\frac{1}{2}}[\,x_0^{+}(n+1),\,x_0^{+}(n)\,]_{q^{2}},\\
%\end{split}
%\end{equation*}
%where we have used \eqref{b26}.  It yields that $[\,x_0^{+}(n+1),\,x_0^{+}(n)\,]_{q^{2}}=0$.
%
%
%
%To check \eqref{b28}, note that for $-n+1\leqslant l_1\leqslant n-1$, $$A_5=[\,x_0^{+}(n),\,x_0^{+}(l_1)\,]_{q^{2}}+[\,x_0^{+}(l_1+1),\,x_0^{+}(n)\,]_{q^{2}}.$$
%It holds that $A_5=0$ by inductive hypothesis. Therefore we have that,
%\begin{equation*}
%\begin{split}
%0&=[\,a_0(1),\,A_5\,]\\
%&=\Big[\,a_0(1),\,[\,x_0^{+}(n),\,x_0^{+}(l_1)\,]_{q^{2}}+[\,x_0^{+}(l_1+1),\,x_0^{+}(n)\,]_{q^{2}}\,\Big]\\
%&=[2]\gamma^{-\frac{1}{2}}\Big([\,x_0^{+}(n+1),\,x_0^{+}(l_1)\,]_{q^{2}}
%+[\,x_0^{+}(l_1+1),\,x_0^{+}(n)\,]_{q^{2}}\Big),
%\end{split}
%\end{equation*}
%which yields \eqref{b28}.

In order to check \eqref{b29}, we consider that for $-n\leq l_2\leq n$,
\begin{align*}
&[\,x_{0}^+(n+1),\,x_0^-(l_2)\,]\\
&=[2]^{-1}\gamma^{\frac{1}{2}}\,\Bigl(\Big[\,\bigl[\,a_0(1),\,x_{0}^-(l_2)\,\bigr],\,
x_{0}^+(n)\,\Big]
+\Big[\,a_0(1),\,\bigl[\,x_{0}^+(n),\,x_{0}^-(l_2)\,\bigr]\,\Big]\,\Bigr)\\
&=-\gamma\,\bigl[\,x_0^-(l_2+1),\,x_{0}^+(n)\,\bigr]=\frac{\gamma^{\frac{n+1-l_2}{2}}\phi_0(n+l_2+1)}{q-q^{-1}}.
\end{align*}

 For \eqref{b30}, one has that,
\begin{equation*}
\begin{split}
&[\,x_0^+(n+1),\,x_0^-(-n-1)\,]\\
%&=-[2]^{-2}\,\Big[\,\bigl[\,a_0(1),\,x_0^+(n)\,\bigr],\,\bigl[\,a_0(-1),\,x_0^-(-n)\,\bigr]\,\Big]\\
&=-[2]^{-2}\Bigl(\Big[\bigl[a_0(1), [x_0^+(n), a_0(-1)] \bigr],\,x_0^-(-n)\Big]
+\Bigl[a_0(-1), \bigl[[a_0(1), x_0^-(-n)], x_0^+(n)\,\bigr]\Bigr]\Bigr)\\
%&=-[2]^{-1}\Bigl(\gamma^{-\frac{1}{2}}\Big[\bigl[a_0(1), x_0^+(n-1)\bigr], x_0^-(-n)\Big]
%+\gamma^{\frac{1}{2}}\Big[a_0(-1), \bigl[x_0^-(-n+1), x_0^+(n)\,\bigr]\Big]\Bigr)\\
%&=\Bigl(\gamma^{-1}\frac{\gamma^n K_0-\gamma^{-n}K_0^{-1}}{q-q^{-1}}+\gamma^nK_0\frac{\gamma-\gamma^{-1}}{q-q^{-1}}\Bigr)\\
&=\frac{\gamma^{n+1} K_0-\gamma^{-n-1} K_0^{-1}}{q-q^{-1}}.
\end{split}
\end{equation*}
\end{proof}

Denote that $\bar{\phi}_0(k)=\frac{K_0^{-1}}{q-q^{-1}}\phi_0(k)$ and $\bar{{\varphi}}_0(-k)=\frac{K_0}{q-q^{-1}}{\varphi}_0(-k)$, we have the following relations.
\begin{prop} \label{p11'} The following relations hold for $d=\gamma^{-\frac{1}{2}}q^2$,
\begin{align}
&[\,\bar{\phi}_0(r),\,x_{0}^{+}(m)\,]=[2]\gamma^{-\frac{1}{2}}\Bigl(\sum\limits_{t=1}^{r-1}(q-q^{-1}){d}^{t-1}x_{0}^{+}(m+t)\bar{\phi}_0(r-t)
+d^{r-1}x_{0}^{+}(r+m)\Bigr),\\
&[\,\bar{{\varphi}}_0(-r),\,x_{0}^{-}(-m)\,]\\\nonumber
&\hskip1.5cm=-[2]\gamma^{-\frac{1}{2}}\Bigl(\sum\limits_{t=1}^{r-1}(q-q^{-1}){d}^{t-1}\bar{{\varphi}}_0(-r+t)x_{0}^{-}(-m-t)
+d^{r-1}x_{0}^{-}(-r-m)\Bigr).
\end{align}
\end{prop}
\begin{proof} By induction, it is enough to check the case of $m=-1$, other cases are similar.
\begin{align*}
&[\,\bar{\phi}_0(n+1),\,x_{0}^{+}(-1)\,]=\gamma^{\frac{1-n}{2}}K_0^{-1}\Bigl(\bigl[\,x_0^+(n),[x_0^-(1),\,x_{0}^{+}(-1)]\bigr]_{q^{2}}+
\bigl[[\,x_0^+(n),\,x_{0}^{+}(-1)]_{q^{2}},x_0^-(1)\bigr]\Bigr)\\
%&=\gamma^{-\frac{n+1}{2}}[2]x_{0}^{+}(n)-\gamma^{-\frac{n+1}{2}}[2]x_{0}^{+}(n)
%-\gamma^{-\frac{1}{2}}\frac{K_0^{-1}}{q-q^{-1}}[\,x_0^+(0),\phi_0(n)]_{q^{2}}\\
%&=d\,[\,\bar{\phi}_0(n),\,x_0^+(0)]_{q^{-4}}\\
&=d\,\Bigl([\,\bar{\phi}_0(n),\,x_0^+(0)]+(q-q^{-1})q^{-2}[2]x_0^+(0)\bar{\phi}_0(n)\Bigl)=\cdots\\
%&=d^2\,\Bigl([\,x_0^+(0),\bar{\phi}_0(n)]+(q-q^{-1})q^{-2}[2]x_0^+(1)\bar{\phi}_0(n-1)\Bigr)\\
%&\hskip1cm+d(q-q^{-1})q^{-2}[2]x_0^+(0)\bar{\phi}_0(n)\\
&=d^n\,\Bigl([\,\bar{\phi}_0(1),\,x_0^+(n-1)]+(q-q^{-1})q^{-2}[2]x_0^+(n-1)\bar{\phi}_0(1)\Bigl)\\
&\hskip1cm+\sum\limits_{t=1}^{n-1}{d}^t(q-q^{-1})q^{-2}[2]x_0^+(t-1)\bar{\phi}_0(n+1-t)\\
%&=\sum\limits_{t=1}^{n}{d}^t(q-q^{-1})q^{-2}[2]x_0^+(t-1)\bar{\phi}_0(n+1-t)+d^n[2]\gamma^{-\frac{1}{2}}x_0^+(n)\\
&=[2]\gamma^{-\frac{1}{2}}\sum\limits_{t=1}^{n}{d}^{t-1}(q-q^{-1})x_0^+(t-1)\bar{\phi}_0(n+1-t)+d^nx_0^+(n).
\end{align*}
\end{proof}

Now we convert the relations to those with $a_0(n+1)$. First it is easily seen that,
\begin{align}
\label{a13}
 &ka_0(k)=k\bar{\phi}_0(k)-(q-q^{-1})\sum\limits_{t=1}^{k-1}t\bar{\phi}_0(k-t)a_0(t),\\
 &ka_0(-k)=k\bar{{\varphi}}_0(-k)-(q-q^{-1})\sum\limits_{t=1}^{k-1}ta_0(-t)\bar{{\varphi}}_0(-k+t).
\end{align}

\begin{prop} \label{p11} The following relations hold for $\epsilon=\pm$ or $\pm1$,
\begin{align}\label{b31}
&[\,a_{0}(\epsilon(n+1)),\,x_{0}^{\epsilon}(-\epsilon)\,]=\frac{[2(n+1)]}{n+1}\gamma^{-\frac{n+1}{2}}x_{0}^{\epsilon}(\epsilon n),\\\label{b32}
&[\,a_{0}(\epsilon(n+1)),\,x_{0}^{-\epsilon}(-\epsilon n)\,]=\frac{[2(n+1)]}{n+1}\gamma^{-\frac{n+1}{2}}x_{0}^{-\epsilon}(\epsilon),\\\label{b33}
&[\,a_{0}(n+1),\,a_0(-n-1)\,]=\frac{[2(n+1)]}{n+1}\frac{\gamma^{n+1}-\gamma^{-(n+1)}}{q-q^{-1}}.
\end{align}
\end{prop}

\begin{proof} %We also use the case of $\epsilon=+$ to show how to prove these relations.
Note that \eqref{b32} is similar to \eqref{b31}, so
we only worry about \eqref{b31}. It follows from Proposition \ref{p11'} that
\begin{equation*}
\begin{split}
&[\,a_{0}(n+1),\,x_{0}^{+}(-1)\,]\\
&=\bigl[\bar{\phi}_0(n+1),\,x_{0}^{+}(-1)\bigr]-\frac{(q-q^{-1})}{n+1}\sum\limits_{t=1}^{n}t\bigl[\bar{\phi}_0(n+1-t)a_0(t),\,x_{0}^{+}(-1)\bigr]\\
&=\frac{[2(n+1)]}{(n+1)!}\gamma^{-\frac{n+1}{2}}x_{0}^{+}(n).
\end{split}
\end{equation*}

For \eqref{b33}, it follows from \eqref{b31}-\eqref{b32} and \eqref{a7} by inductive hypothesis that
\begin{align*}
&[\,a_{0}(n+1),\,a_{0}(-n-1)\,]\\
&=\gamma^{\frac{n-1}{2}}K_0[\,a_{0}(n+1),\,[\,x_0^+(-1),\,x_0^-(-n)\,]]-\sum\limits_{\substack{\lambda\vdash n+1\\ \lambda\neq(n+1)}}\frac{(q-q^{-1})^{l(\lambda)-1}}{m_{\lambda}!}\bigl[\,a_{0}(n+1),\,a_0(\lambda)\bigr]\\
&=\gamma^{\frac{n-1}{2}}K_0\bigl([\,[a_{0}(n+1),\,x_0^+(-1)],\,x_0^-(-n)\,]
+[\,x_0^+(-1),\,[a_{0}(n+1),\,x_0^-(-n)\,]\bigr)\\
&=\frac{[2(n+1)]}{n+1}\frac{\gamma^{n+1}-\gamma^{-(n+1)}}{q-q^{-1}}.
\end{align*}
\end{proof}

From Proposition \ref{p10} and Proposition \ref{p11}, we have checked that all  elements $x_0^{\pm}(\epsilon(n+1))$, $a_0(\epsilon(n+1))$  satisfy the defining relations \eqref{e:comm0}--\eqref{e:comm0c} and \eqref{e:Dr3}--\eqref{e:Dr11}. By induction it follows that
all  elements $x_0^{\pm}(\epsilon k),\,a_0(\epsilon k)$ for $k\in \mathbb{Z}/\{0\}$ satisfy the defining relations
\eqref{e:comm0}-\eqref{e:comm0c} and \eqref{e:Dr3}--\eqref{e:Dr11}. In other words, the algebra ${\mathcal U}_0$ contains a subalgebra $U_q(\hat{sl_2})_0$ isomorphic to $U_q(\hat{sl_2})$.

Next we will show that the quantum toroidal algebra ${\mathcal U}_0$ also contains another 
subalgebra $U_q(\hat{sl_2})_1$ isomorphic to $U_q(\hat{sl_2})$.
i.e. the subalgebra generated by $x_1^{+}(-1)$, $x_1^{-}(1)$ etc satisfy exactly the same relations as those of $x_0^{+}(-1)$, $x_0^{-}(1)$
(cf. Prop. \ref{p13}). Moreover, we will also derive the relations between $U_q(\hat{sl_2})_0$ and $U_q(\hat{sl_2})_1$.

For $\epsilon=\pm1$ or $\pm$, we define that
\begin{gather}\label{a15}
x_1^{\pm}(\epsilon)=\pm\gamma^{\pm\frac{1}{2}}\,\bigl[\,a_0(\epsilon),\,x_1^{\pm}(0)\,\bigr]\in
{\mathcal U}_0,\\\label{a16}
a_1(1)=\gamma^{\frac{1}{2}}K_1^{-1}\,\bigl[\,x_{1}^+(0),\,x_1^{-}(1)\,\bigr]\in
{\mathcal U}_0,\\\label{a17}
a_1(-1)=\gamma^{-\frac{1}{2}}K_1\,\bigl[\,x_{1}^+(-1),\,x_1^{-}(0)\,\bigr]\in
{\mathcal U}_0.
\end{gather}

Similar to Lemma \ref{p2}, we have the following relations.

\begin{lemm}\label{l5} It is easy to see that for $\epsilon=\pm 1$,
\begin{align}\label{bb16}
&\bigl[\,x_1^+(\epsilon),\,x_1^-(0)\,\bigr]=\gamma \bigl[\,x_1^+(0),\,x_1^-(\epsilon)\,\bigr],\\
%\label{bb17}
%&\bigl[\,x_1^+(0),\,x_1^-(-1)\,\bigr]=\gamma^{-1} \bigl[\,x_1^+(-1),\,x_1^-(0)\,\bigr]\\
&\bigl[\,a_1(1), a_1(-1)\,\bigr]=[2]\frac{\gamma-\gamma^{-1}}{q-q^{-1}},\\
\label{bb18}
&a_1(1)=K_1^{-1}\gamma^{\frac{1}{2}}\bigl[\,x_1^+(0),\,x_1^-(1)\,\bigr]=K_1^{-1}\gamma^{-\frac{1}{2}} \bigl[\,x_1^+(1),\,x_1^-(0)\,\bigr],\\
\label{bb19}
&a_1(-1)=K_1\gamma^{-\frac{1}{2}}\bigl[\,x_1^+(-1),\,x_1^-(0)\,\bigr]=K_1\gamma^{\frac{1}{2}} \bigl[\,x_1^+(0),\,x_1^-(-1)\,\bigr].
\end{align}
\end{lemm}

%\begin{proof}\, \eqref{bb16} and \eqref{bb17} is almost same to that of Lemma  \ref{p2}, which can be checked directly.
%\end{proof}

\begin{prop}\label{p13} From the above constructions, we have the following relations, which are compatible to the defining relations of the quantum toroidal algebra $U_q(sl_{n+1},tor)$ for $\epsilon=\pm$ or $\pm1$.
\begin{align}\label{b34}
&K_ix_1^{\pm}(\epsilon)K_i^{-1}=q^{\pm a_{i1}}x_1^{\pm}(1),\quad
q^{d_1}x_1^{\pm}(\epsilon)q^{-d_1}=q^{\epsilon} x_1^{\pm}(1),\quad q^{d_2}x_1^{\pm}(\epsilon)q^{-d_2}=q^{\epsilon} x_1^{\pm}(1),\\\label{b35}
&\bigl[\,a_0(1), a_1(-1)\,\bigr]=-\frac{\gamma-\gamma^{-1}}{q-q^{-1}},\\\label{b36}
&[\,x_1^{\pm}(1),\,x_0^{\pm}(0)\,]_{q^{\mp1}}+[\,x_0^{\pm}(1),\,x_1^{\pm}(0)\,]_{q^{\mp1}}=0,
\\\label{b37}
&[\,x_1^{\epsilon}(-\epsilon),\,x_1^{\epsilon}(0)\,]_{q^{-2}}=0,\\\label{b38}
&[\,a_0(\epsilon),\,x_1^{\pm}(-\epsilon)\,]=-\gamma^{\mp\frac{1}{2}}x_1^{\pm}(0),\\
%\end{align}
%\begin{align}
\label{b39}
&[\,x_i^{-\epsilon}(0),\,x_1^{\epsilon}(-\epsilon)\,]=0, \ \ \mbox{for} \ \ i\neq 1, \\\label{b40}
&[\,x_1^+(-1),\,x_1^-(1)\,]=\frac{\gamma^{-1} K_1-\gamma K_1^{-1}}{q-q^{-1}},\\\label{b41}
&\bigl[\,x_0^{\pm}(0),\,[x_0^{\pm}(0),\,x_1^{\pm}(1)]_{q^{-1}}\bigr]_q=0.
\end{align}
\end{prop}

\begin{proof}\, Using the Chevalley anti-involution $\iota$, we only check for the case of $\epsilon=+$.  \eqref{b34} holds directly by definition. \eqref{b35} follows from \eqref{a17}.
%\begin{equation*}
%\begin{split}
%&[\,a_0(1), a_1(-1)\,]\\&=K_1\gamma^{-\frac{1}{2}}[\,a_0(1),\,[x_1^+(-1),
%x_1^-(0)],\,]\\
%&=K_1\gamma^{-\frac{1}{2}}\Bigl(\bigl[\,[\,a_0(1),\,x_1^+(-1)\,],
%x_1^-(0)\,\bigr]+\bigl[\,x_1^+(-1),\,[a_0(1),\,
%x_1^-(0)\,]\,\bigr]\Big)\\
%&=-pK_\gamma^{-\frac{1}{2}}\Bigl(\gamma^{-1/2}\bigl[\,x_1^+(0),\,
%x_1^-(0)\,\bigr]-\gamma^{1/2}\bigl[\,x_1^+(-1),\,x_1^-(1)\,\bigr]\Big)\\
%&=-p\frac{\gamma-\gamma^{-1}}{q-q^{-1}}.
%\end{split}
%\end{equation*}
\eqref{b36} can be easily checked as follows. %which can be checked directly.
\begin{equation*}
\begin{split}
&[\,x_1^-(1),\,x_0^-(0)\,]_q=\gamma^{-\frac{1}{2}}\Big[\,\bigl[\,a_0(1),\,x_1^-(0)\,\bigr],\,x_0^-(0)\,\Big]_q\\
&=K_0^{-1}\Big[\,\bigl[\,[x_0^+(0),\,x_0^-(1)],\,x_1^-(0)\,\bigr]_q,\,x_0^-(0)\,\Big]_{q^{-1}}\\
&=K_0^{-1}\Bigl(\Big[\bigl[x_0^+(0),[x_0^-(1),x_1^-(0)]_q\,\bigr]_1,x_0^-(0)\Big]_{q^{-1}}
{+}q\Big[\bigl[[x_0^+(0), x_1^-(0)],x_0^-(1)\,\bigr]_q, x_0^-(0)\Big]_{q^{-1}}\Bigr)\\
&=K_0^{-1}\Bigl(\Big[x_0^+(0),\bigl[[x_0^-(1),x_1^-(0)]_q, x_0^-(0)\bigr]_{q^{-1}}\Big]
{+}q^{-1}\Big[\bigl[[x_0^+(0), x_0^-(0)], [x_0^-(1), x_1^-(0)\bigr]_{q}\,\Big]_q\Bigr)\\
&=-[\,x_0^-(1),\,x_1^-(0)\,]_q,
\end{split}
\end{equation*}
which is compatible with the defining relation \eqref{e:Dr3}. %\eqref{e:comm3}.
\eqref{b37} follows from \eqref{a15} and \eqref{a2} by using \eqref{f:comm4} and \eqref{f:comm6}.
%\begin{equation*}
%\begin{split}
%&[\,x_1^+(-1),\,x_1^+(0)\,]_{q^{-2}}\\
%&=-pK_0\Big[\bigl[\,[x_0^+(-1), x_0^-(0)], x_1^+(0)\bigr]_{q}, x_1^+(0)\Big]_{q^{-1}}\\
%&=-pK_0\Bigl(\bigl[[x_0^+(-1), [x_0^-(0), x_1^+(0)]]_{q}, x_1^+(0)\bigr]_{q^{-1}}
%+\bigl[[x_0^+(-1),\,x_1^+(0)]_{q}, x_0^-(0)], x_1^+(0)\bigr]_{q^{-1}}\Bigr)\\
%&=-pK_0\Bigl(\bigl[[x_0^+(-1),\,x_1^+(0)]_{q}, [x_0^-(0), x_1^+(0)]\bigr]_{q^{-1}}
%{+}\bigl[[[x_0^+(-1), x_1^+(0)]_{q}, x_1^+(0)]_{q^{-1}}, x_0^-(0)\bigr]_{q^{-1}}\Bigr)\\
%&=0,
%\end{split}
%\end{equation*}
%where we have used  \eqref{f:comm4} and \eqref{f:comm6}.
\eqref{b38}-\eqref{b40} are easy, which can be verified similarly.
Relation \eqref{b41} is compatible with the defining relation \eqref{e:Dr11} for $a_{ij}=-1$. %\eqref{e:comm5}.
It follows from \eqref{a15} that
\begin{equation*}
\begin{split}
&\bigl[\,x_0^-(0),\,[x_0^-(0),\,x_1^-(1)]_{q^{-1}}\bigr]_q\\
%&=p^{-1}\gamma^{-\frac{1}{2}}\bigl[\,x_0^-(0),\,[x_0^-(0),\,[a_0(1),\,x_1^-(0)]\,]_{q^{-1}}\bigr]_q\\
&=\gamma^{-\frac{1}{2}}\Bigl(\bigl[\,x_0^-(0),\,[[x_0^-(0),\,a_0(1)],\,x_1^-(0)\,]_{q^{-1}}\bigr]_q
+\bigl[\,x_0^-(0),\,[a_0(1),\,[x_0^-(0),\,x_1^-(0)]_{q^{-1}}]_1\bigr]_q\Bigr)\\
&=[2]\bigl[\,x_0^-(0),\,[x_0^-(1)\,x_1^-(0)\,]_{q^{-1}}\bigr]_q
+\gamma^{-\frac{1}{2}}\bigl[\,[x_0^-(0),\,a_0(1)]_1,\,[x_0^-(0),\,x_1^-(0)]_{q^{-1}}\bigr]_q\\
&\hskip6.0cm+\gamma^{-\frac{1}{2}}\bigl[\,[a_0(1),\,[x_0^-(0),\,[x_0^-(0),\,x_1^-(0)]_{q^{-1}}]_q\bigr]_1\\
&=[2]\Bigl(\bigl[\,x_0^-(0),\,[x_0^-(1)\,x_1^-(0)\,]_{q^{-1}}\bigr]_q
+\bigl[\,x_0^-(1),\,[x_0^-(0),\,x_1^-(0)]_{q^{-1}}\bigr]_q\Bigr)=0,
\end{split}
\end{equation*}
where we have used the Serre relation \eqref{f:comm6}.

\end{proof}

Similarly we can construct the subalgebras $U_q(\hat{sl_2})_i$, $i=2, \ldots, n$,
that is, we can obtain all other Drinfeld generators $x_i^{\pm}(k), a_i(r)$
and verify the interrelations among them.

On the other hand, as the Dynkin diagram of affine Lie algebra $\hat{\mathfrak{sl}_{n+1}}$ is a cycle, for $\epsilon=\pm1$ or $\pm$, we define that
\begin{gather}
\dot{y}_n^{\pm}(\epsilon)=\pm \gamma^{\pm\frac{1}{2}}\,\bigl[\,a_0(\epsilon),\,x_n^{\pm}(0)\,\bigr],\\
\dot{b}_n(1)=\gamma^{\frac{1}{2}}K_1^{-1}\,\bigl[\,x_{n}^+(0),\,\dot{y}_n^{-}(1)\,\bigr],\\
\dot{b}_n(-1)=\gamma^{-\frac{1}{2}}K_1\,\bigl[\,\dot{y}_{n}^+(-1),\,x_n^{-}(0)\,\bigr].
\end{gather}

It is easy to see that $x_n^{-\epsilon}(\epsilon)-\dot{y}_n^{-\epsilon}(\epsilon) \in Ker \pi$ for  $\epsilon=\pm1$ or $\pm$.
Therefore
by induction on degree we have proved Theorem \ref{newpre}.

%\begin{theo}\label{T:newgen}\, As associative algebras, ${\bar{\mathcal U}}_0\simeq U_q(sl_{n+1},tor)$.
%\end{theo}

\begin{remark}\, The quantum toroidal algebra contains two subalgebras $\mathcal{A}_1$ and $\mathcal{A}_2$ isomorphic
to the quantum affine algebra $U_q(\hat{sl}_n)$. These two subalgebras
are generated by Serre generators and partial Drinfeld generators respectively. %similar type of generators.
 $$
\mathcal{A}_1:=\left.\left\langle\, x_i^{\pm}(0), \,K_i^{\pm1},\,q^{d_1},\,
\gamma^{\pm\frac{1}2}\; \right| i\in I\;\right\rangle,$$
$$
\mathcal{A}_2:=\left.\left\langle\, x_i^{\pm}(0),\,  x_0^+(-1),\,x_0^-(1), \,K_i^{\pm1},\,q^{d_2},\,
\gamma^{\pm\frac{1}2}\; \right| i\in I/\{n\}\;\right\rangle.$$
It is easy to see that the subalgebra $\mathcal{A}_1$ is isomorphic to the Drinfeld-Jimbo realization of the quantum affine algebra $U_q(\hat{sl}_{n+1})$, and $\mathcal{A}_2$ is isomorphic to its Drinfeld realization.
\end{remark}

Denote by $\mathcal{U}_1$ the subalgebra of
$U_q(sl_{n+1},tor)$ generated by
$ x_j^{\pm}(0), \,x_1^+(1),\, x_1^-(-1)$,  $x_0^+(-1)$, $x_0^-(1)$, $K_i^{\pm1}$ ($j\in I/\{0\},\, i\in I$), $q^{d_1},\,q^{d_2}$
 and $\gamma^{\pm\frac{1}2}$ satisfying the relations \eqref{e:comm0}-\eqref{e:comm0c} and \eqref{e:Dr3}-\eqref{e:Dr11}, that is,
$$
{\mathcal U}_1:=\left.\left\langle\, x_j^{\pm}(0), \,x_1^+(1),\, x_1^-(-1),\, x_0^+(-1),\,x_0^-(1), \,K_i^{\pm1},\,q^{d_1},\,q^{d_2},\,
\gamma^{\pm\frac{1}2}\; \right|i\in I, j\in I/\{0\}\;\right\rangle.
$$
It then can be shown that ${\mathcal U}_1$ and ${\mathcal U}_0$ are isomorphic.

%\begin{prop}\, ${\mathcal U}_1\cong {\mathcal U}_0$
%\end{prop}
%\begin{proof}\, By Proposition 2.1, it suffices to show that the quantum affine algebra $\mathcal{A}_1$ is
%generated by $x_i^{\pm}(0), \,x_1^+(1),\, x_1^-(-1),\, K_i^{\pm1},\,
%\gamma^{\pm\frac{1}2}$  for $i\in I/\{0\}$, which follows from the result of \cite{JZ3}.
%\end{proof}

\section{Hopf algebra structure of ${\mathcal U}_0$}
In this section we show that there is a similar Hopf algebra structure on ${\mathcal U}_0$ for $n=1, 2$ extending from its vertical and horizontal quantum affine subalgebras, and conjecture that ${\mathcal U}_0$ for general $n$ is also a Hopf algebra under our formulas.

First we define the actions of $\Delta$ on $x_0^-(1)$ and $x_0^+(-1)$ for the case of $n=1$ as follows:
\begin{gather*}
\begin{split}
&\Delta(x_0^+(-1))=x_0^+(-1)\ot 1+\gamma K_0^{-1}\ot x_0^+(-1)+(q-q^{-1})\gamma K_0^{-1}x_1^-(0)\ot [x_0^+(-1),\,x_1^+(0)]_q,\\
&\Delta(x_0^-(1))=x_0^-(1)\ot\gamma^{-1}K_0+1\ot x_0^-(1)-(q-q^{-1})[x_{1}^-(0),\,x_0^-(1)]_{q^{-1}}\ot x_1^+(0)\gamma^{-1} K_0.
\end{split}
\end{gather*}

The actions of $\Delta$ on $x_0^-(1)$ and $x_0^+(-1)$ for the case of $n=2$  is given as follows:
\begin{gather*}
\begin{split}
&\Delta(x_0^+(-1))=x_0^+(-1)\ot 1+\gamma K_0^{-1}\ot x_0^+(-1)+(q-q^{-1})\gamma K_0^{-1}x_1^-(0)\ot [x_0^+(-1),\,x_1^+(0)]_q\\
&\hskip1.9cm+(q-q^{-1})\gamma K_0^{-1}[x_2^-(0), x_1^-(0)]_{q^{-1}}\ot  [[x_0^+(-1),\,x_1^+(0)]_q,x_2^+(0)]_{q},\\
&\Delta(x_0^-(1))=x_0^-(1)\ot\gamma^{-1}K_0+1\ot x_0^-(1)-(q-q^{-1})[x_{1}^-(0),\,x_0^-(1)]_{q^{-1}}\ot x_1^+(0)\gamma^{-1}K_0 \\
&\hskip1.6cm -(q-q^{-1})[x_2^-(0),[x_{1}^-(0),\,x_0^-(1)]_{q^{-1}}]_{q^{-1}}\ot [x_1^+(0),\, x_{2}^+(0)]_{q}\gamma^{-1}K_0.
\end{split}
\end{gather*}

\begin{defi}\,Under the above action of $\Delta$,  we
now define a comultiplication map $\Delta: {\mathcal U}_0\longrightarrow {\mathcal U}_0\otimes
{\mathcal U}_0$  for $n=1$ and $n=2$ on the generators as follows.
\begin{align}\label{e:commult1}
&\Delta(K_i)=K_i\ot K_i, \qquad \Delta(\gamma^{\pm\frac{1}2})=\gamma^{\pm\frac{1}2}\otimes
\gamma^{\pm\frac{1}2},\qquad
\Delta(q^{d_1})=q^{d_1}\ot q^{d_1}, \\ \label{e:commult2}
%\Delta(q^{d_2})=q^{d_2}\ot q^{d_2}, \\
&\Delta(x_i^+(0))=x_i^+(0)\ot 1+K_i\ot x_i^+(0), \qquad
\Delta(x_i^-(0))=x_i^-(0)\ot K_i^{-1}+1\ot
x_i^-(0).
\end{align}
\end{defi}

\begin{prop}\label{p:5.4} \,Under the above map $\Delta$, the algebra ${\mathcal U}_0$\, for $n=1$ or $n=2$ becomes a bialgebra.
%That is, $\Delta$ defines a morphism of coalgebra from $\mathcal{U}_0$  into $\mathcal{U}_0\otimes\mathcal{U}_0$.
\end{prop}

\begin{proof}\,  The proof will be divided into  three cases $n=1$ and $n=2$  of ${\mathcal U}_0$.

(I) For the case of $U_q(sl_{2},tor)$, %by the construction
%$$\Delta(x_0^-(0))=x_0^-(0)\ot K_0^{-1}+1\ot x_0^-(0),$$
%and
%$$\Delta(x_0^-(1))=x_0^-(1)\ot\gamma^{-1}K_0+1\ot x_0^-(1)-(q-q^{-1})[x_{1}^-(0),\,x_0^-(1)]_{q^{-1}}\ot x_1^+(0)\gamma^{-1} K_0$$
we focus on checking that $ [\Delta (x_0^-(1)),\,\Delta(x_0^-(0))]_{q^{-2}}=0.$
\begin{eqnarray*}
&&[\Delta(x_0^-(1)),\,\Delta(x_0^-(0))\,]_{q^{-2}}\\
&=&[\,x_0^-(1),\,x_0^-(0)\,]_{q^{-2}}\ot \gamma^{-1}-(q-q^{-1})\bigl[[\,x_1^-(0),\,x_0^-(1)\,]_{q^{-1}},\,x_0^-(0)\bigr]_{q^{-1}}\ot x_1^+(0)\gamma^{-1}\\
&&+1\ot[\,x_0^-(1),\,x_0^-(0)\,]_{q^{-2}} -(q-q^{-1})q^{-2}[x_0^-(0),\,x_0^-(1)]_{q^{-1}}\ot [x_1^+(0),\,x_0^-(0)]\gamma^{-1}K_0=0,
\end{eqnarray*}
where the first and third item are killed for \eqref{f:comm3}, and the second one is zero for \eqref{f:comm5}, the last item vanishes for
\eqref{f:comm4}.

Next, we further check that the action of $\Delta$ keeps the relation \eqref{f:comm4}. %By the construction,
%
%
%$$\Delta(x_1^+(0))=x_1^+(0)\ot 1+K_1\ot x_1^+(0),$$ one gets that,
%
\begin{eqnarray*}
&&[\,\Delta(x_1^+(0)),\,\Delta(x_0^-(1))\,]\\
&=&[\,x_1^+(0),\,x_0^-(1)\,]\ot \gamma^{-1} K_0+K_1\ot[\,x_1^+(0),\,x_0^-(1)\,]+(q-q^{-1})x_0^-(1)K_1\ot x_1^+(0)\gamma^{-1}K_0\\
&&-(q-q^{-1})x_0^-(1)K_1\ot x_1^+(0)\gamma^{-1}K_0=0.
\end{eqnarray*}

On the other hand, it is easy to see that
\begin{eqnarray*}
&&[\,\Delta(x_0^+(-1)),\,\Delta(x_0^-(1))\,]\\
&=&[\,x_0^+(-1),\,x_0^-(1)\,]\ot \gamma^{-1}K_0+\gamma K_0^{-1}\ot[\,x_0^+(-1),\,x_0^-(1)\,]\\
&&-(q-q^{-1}) \big[x_1^-(0),\,[x_0^+(-1),\,x_0^-(1)]\big]_{q^{-1}}\ot x_0^+(-1)\gamma^{-1}K_0\\
&&+(q-q^{-1})\gamma K_0^{-1}x_1^-(0)\ot \big[[x_0^+(-1),\,x_0^-(1)],\,x_1^+(0)\big]_q\\
&&-(q-q^{-1}) \big[\gamma K_0^{-1}\ot x_0^+(-1),\,[x_1^-(0),\,x_0^-(1)]_{q^{-1}}\ot x_1^+(0)\gamma^{-1}K_0\big]\\
&&+(q-q^{-1})\big[\gamma K_0^{-1}x_1^-(0)\ot [x_0^+(-1),\,x_1^+(0)]_q,\, x_1^-(0)\ot\gamma^{-1}K_0\big]\\
&&-(q-q^{-1})^2 \big[\gamma K_0^{-1} x_1^-(0)\ot [x_0^+(-1),\,x_1^+(0)]_q,\, [x_1^-(0),\, x_{0}^-(1)]_{q^{-1}}\ot x_1^+(0)\gamma^{-1}K_0\big].
\end{eqnarray*}
By direct computation using \eqref{e:Dr8}%\eqref{e:comm4}
and \eqref{e:Dr3}, %\eqref{e:comm1},
the third and forth terms cancel each other, and the fifth term vanishes.
 As for the last term, we calculate by \eqref{e:Dr3} %\eqref{e:comm1}
 and the Serre relations as follows.
\begin{align*}
&-(q-q^{-1})^2 \big[\gamma K_0^{-1} x_1^-(0)\ot [x_0^+(-1),\,x_1^+(0)]_q,\, [x_1^-(0),\, x_{0}^-(1)]_{q^{-1}}\ot x_1^+(0)\gamma^{-1}K_0\big]\\
&=-(q-q^{-1})^2 \gamma K_0^{-1}\underbrace{\big[ x_1^-(0),\,[x_1^-(0),\, x_{0}^-(1)]_{q^{-1}}\big]_q}\ot [x_0^+(-1),\,x_1^+(0)]_q \cdot x_1^+(0)\gamma^{-1}K_0\\
&-(q-q^{-1})^2 q \gamma K_0^{-1} x_1^-(0)\cdot[x_1^-(0),\, x_{0}^-(1)]_{q^{-1}}\ot \underbrace{\big[[x_0^+(-1),\,x_1^+(0)]_q,\, x_1^+(0)\big]_{q^{-1}}}\gamma^{-1}K_0=0.
\end{align*}

Therefore, one has that $[\Delta(x_0^+(-1)),\,\Delta(x_0^-(1))\,]
%&=[\,x_0^+(-1),\,x_0^-(1)\,]\ot \gamma^{-1}K_0+\gamma K_0^{-1}\ot[\,x_0^+(-1),\,x_0^-(1)\,]\\
=\frac{\Delta(\gamma^{-1} K_0)-\Delta(\gamma K_0^{-1})}{q-q^{-1}}.$

Finally, we would like to check  Serre relation
$[\,\Delta(x_1^-(0)),\,[\Delta(x_1^-(0)),\,\Delta(x_0^-(1))]_q\,]_{q^{-1}}=0$.

In fact, it follows from the construction of $\Delta(x_1^-(0))$ and $\Delta(x_0^-(1))$ that
\begin{align*}
&[\,\Delta(x_1^-(0)),\,[\Delta(x_1^-(0)),\,\Delta(x_0^-(1))]_q^{-1}\,]_{q}\\
%&=[\,\Delta(x_1^-(0)),\,1\ot [\,x_1^-(0),\,x_0^-(1)\,]_{q^{-1}}+[\,x_1^-(0),\,x_0^-(1)\,]_{q^{-1}}\ot \gamma^{-1}K_0K_1]\\
&=[x_1^-(0),\,[\,x_1^-(0),\,x_0^-(1)\,]_{q^{-1}}]_q\ot\gamma^{-1}K_0+1\ot [x_1^-(0),\,[\,x_1^-(0),\,x_0^-(1)\,]_{q^{-1}}]_q=0.
\end{align*}

(II) For the case of $n=2$, %the action of comultiplication $\Delta$ on the generators $x_0^-(1)$ and $x_0^+(-1)$ can be defined as follows:
similarly it suffices to check the relations involving $x_0^-(1)$ and $x_0^+(-1)$.
More precisely, we first verify that $[\,\Delta(x_0^-(1)),\,\Delta(x_0^-(0))\,]_{q^{-2}}=0$.

Note that $\Delta(x_0^-(0))=x_0^-(0)\ot K_0^{-1}+1\ot x_0^-(0)$ and some items are killed by \eqref{f:comm4} directly:
\begin{align*}
&[\,\Delta(x_0^-(1)),\,\Delta(x_0^-(0))\,]_{q^{-2}}\\
&=-(q-q^{-1})\bigl[[x_{1}^-(0),\,x_0^-(1)]_{q^{-1}},\,x_0^-(0)\bigl]_{q^{-1}}\ot x_1^+(0)\gamma^{-1}\\
&-(q-q^{-1})\bigl[[x_2^-(0),\, x_{1}^-(0),\,x_0^-(1)]_{(q^{-1},q^{-2})},\,x_0^-(0)\bigl]\ot [x_1^+(0),\,x_2^+(0)]_{q}\gamma^{-1} \\
&+(q-q^{-1})\bigl[[x_{1}^-(0),\,x_2^-(0),\, x_{1}^-(0),\,x_0^-(1)]_{(q^{-1},q^{-2},1)},\,x_0^-(0)\bigl]_{q}\ot [x_1^+(0),\,x_2^+(0),\,x_{1}^+(0)]_{(q, q^{-1})}\gamma^{-1} \\
&-(q-q^{-1})\frac{q}{2}\bigl[[x_2^-(0),\,x_{1}^-(0),\,x_2^-(0),\, x_{1}^-(0),\,x_0^-(1)]_{(q^{-1},q^{-2},1, q^{-1})},\,x_0^-(0)\bigl]_{q^2}\\
&\hskip4.2cm\ot [x_1^+(0),\,x_2^+(0),\,x_{1}^+(0),\,x_2^+(0)]_{(q,q^{-1},q^{-2})}\gamma^{-1}\\
&-(q-q^{-1})\bigl[[x_{2}^-(0),\,x_0^-(1)]_{q^{-1}},\,x_0^-(0)\bigl]_{q^{-1}}\ot x_2^+(0)\gamma^{-1}\\
&-(q-q^{-1})\bigl[[x_1^-(0),\, x_{2}^-(0),\,x_0^-(1)]_{(q^{-1},q^{-2})},\,x_0^-(0)\bigl]\ot [x_2^+(0),\,x_1^+(0)]_{q}\gamma^{-1}\\
&+(q-q^{-1})\bigl[[x_{2}^-(0),\,x_1^-(0),\, x_{2}^-(0),\,x_0^-(1)]_{(q^{-1},q^{-2},1)},\,x_0^-(0)\bigl]_{q}\ot [x_2^+(0),\,x_1^+(0),\,x_{2}^+(0)]_{(q, q^{-1})}\gamma^{-1} \\
&-(q-q^{-1})\frac{q}{2}\bigl[[x_1^-(0),\,x_{2}^-(0),\,x_1^-(0),\, x_{2}^-(0),\,x_0^-(1)]_{(q^{-1},q^{-2},1, q^{-1})},\,x_0^-(0)\bigl]_{q^2}\\
&\hskip4.2cm\ot [x_2^+(0),\,x_1^+(0),\,x_{2}^+(0),\,x_1^+(0)]_{(q,q^{-1},q^{-2})}\gamma^{-1}=0.
\end{align*}
where the terms cancel each other by the Serre relations.

Now we turn to $[\,\Delta(x_1^+(0)),\,\Delta(x_0^-(1))\,]=0$.
Note that $\Delta(x_1^+(0))=x_1^+(0)\ot 1+K_1\ot x_1^+(0)$,
the left-hand side can be computed by the Serre relations and \eqref{e:commult3}-\eqref{e:commult4}.
\begin{align*}
&[\,\Delta(x_1^+(0)),\,\Delta(x_0^-(1))\,]\\
&=-(q-q^{-1})[x_{2}^-(0),\,x_0^-(1)]_{q^{-1}}K_1\ot [x_1^+(0),\,x_2^+(0)]_q\gamma^{-1} K_0\\
&-(q-q^{-1})\frac{q}{2}\, [x_2^-(0), x_1^-(0), x_{2}^-(0), x_0^-(1)]_{(q^{-1}, q^{-2}, 1)}K_1\\
&\hskip4.2cm\ot [x_1^+(0), x_2^+(0), x_1^+(0), x_2^+(0)\,]_{(q, q^{-1}, q^{-2})}\gamma^{-1} K_0\\
&-(q^2-q^{-2})[x_{2}^-(0),\,x_0^-(1)]_{q^{-1}}K_1\ot [x_2^+(0),\,x_1^+(0)]_q\gamma^{-1} K_0\\
&-(q-q^{-1})\frac{q}{2}\, [x_2^-(0), x_1^-(0), x_{2}^-(0), x_0^-(1)]_{(q^{-1}, q^{-2}, 1)}K_1\\
&\hskip4.2cm\ot [x_2^+(0), x_1^+(0), x_2^+(0), x_1^+(0)\,]_{(q, q^{-1}, q^{-2})}\gamma^{-1} K_0\\
&-(q-q^{-1})q^2[x_{2}^-(0),\,x_0^-(1)]_{q^{-1}}K_1\ot [x_1^+(0),\,x_2^+(0)]_{q^{-3}}\gamma^{-1} K_0\\
&-(q-q^{-1})q\, [x_2^-(0), x_1^-(0), x_{2}^-(0), x_0^-(1)]_{(q^{-1}, q^{-2}, 1)}K_1\\
&\hskip4.2cm\ot [x_1^+(0), x_2^+(0), x_1^+(0), x_2^+(0)\,]_{(q, q^{-1}, q^{-2})}\gamma^{-1} K_0=0,
\end{align*}
where we have used the following two relations (which are based on the Serre relations):
\begin{align}
&[x_1^+(0), x_1^+(0), x_2^+(0), x_1^+(0), x_2^+(0)\,]_{(q, q^{-1},q^{-2}, 1)}=0,\\
&[x_1^+(0), x_2^+(0), x_1^+(0), x_2^+(0), x_1^+(0)\,]_{(q, q^{-1},q^{-2}, 1)}=0.
\end{align}
Similarly it can be checked that $[\,\Delta(x_2^+(0)),\,\Delta(x_0^-(1))\,]=0.$

Next we need to verify that
\begin{equation*}
[\,\Delta(x_0^+(-1)),\,\Delta(x_0^-(1))\,]=\frac{\Delta(\gamma^{-1} K_0)-\Delta(\gamma K_0^{-1})}{q-q^{-1}}.
\end{equation*}

It goes as follows by the actions of $\delta$ on the generators $x_0^+(-1)$ and $x_0^-(1)$.
\begin{eqnarray*}
&&[\,\Delta(x_0^+(-1)),\,\Delta(x_0^-(1))\,]\\
&=&\Big[\,\Delta(x_0^+(-1)),\,x_0^-(1)\ot\gamma^{-1}K_0+1\ot x_0^-(1)-(q-q^{-1})[x_{1}^-(0),\,x_0^-(1)]_{q^{-1}}\ot x_1^+(0)\gamma^{-1}K_0 \\
&&\hskip3.6cm -(q-q^{-1})[x_2^-(0),[x_{1}^-(0),\,x_0^-(1)]_{q^{-1}}]_{q^{-1}}\ot [x_1^+(0),\, x_{2}^+(0)]_{q}\gamma^{-1}K_0\Big]\\
&=&[\,x_0^+(-1),\,x_0^-(1)\,]\ot \gamma^{-1}K_0+\gamma K_0^{-1}\ot[\,x_0^+(-1),\,x_0^-(1)\,]\\
&&-(q-q^{-1})\big[[x_0^+(-1),\,[x_{1}^-(0),\,x_0^-(1)]_{q^{-1}}]\ot x_1^+(0)\gamma^{-1}K_0\\
&&-(q-q^{-1})\big[x_0^+(-1),\,[x_2^-(0),\,[x_{1}^-(0),\,x_0^-(1)]_{q^{-1}}]_{q^{-1}}\big]\ot [x_1^+(0),\,x_2^+(0)]_{q}\gamma^{-1}K_0\\
&&+(q-q^{-1})\gamma K_0^{-1}x_1^-(0)\ot\big[[x_0^+(-1),\,x_{1}^+(0)]_{q},\,x_0^-(1)\big] \\
&&+(q-q^{-1})\gamma K_0^{-1}[x_2^-(0),\,x_1^-(0)]_{q^{-1}}\ot\big[[[x_0^+(-1),\,x_{1}^+(0)]_{q},\,x_2^+(0)]_{q},\,x_0^-(1)\big].
\end{eqnarray*}

In fact, by direct computation using \eqref{e:Dr8} %\eqref{e:comm4}
and \eqref{e:Dr3} %\eqref{e:comm1}
similar to Case (I), the last terms are killed due to the Serre relations.
We take one term as an example,
\begin{eqnarray*}
&&\gamma K_0^{-1}\Big(x_1^-(0)\cdot[x_1^-(0),\,x_0^-(1)]_{q^{-1}}\ot [x_0^+(-1),\,x_1^+(0)]_{q}\cdot x_1^+(0)\\
&&-[x_1^-(0),\,x_0^-(1)]_{q^{-1}}\cdot x_1^-(0)\ot x_1^+(0)\cdot[x_0^+(-1),\,x_1^+(0)]_{q}\Big)\gamma^{-1}K_0\\
&=&\gamma K_0^{-1}\Big(\underbrace{[x_1^-(0),\,x_1^-(0),\,x_0^-(1)]_{(q^{-1}, q)}}\ot [x_0^+(-1),\,x_1^+(0)]_{q}\cdot x_1^+(0)\\
&&-[x_1^-(0),\,x_0^-(1)]_{q^{-1}}\cdot x_1^-(0)\ot \underbrace{[x_1^+(0),\,x_0^+(-1),\,x_1^+(0)]_{(q,\,q^{-1})}}\Big)\gamma^{-1}K_0=0.
\end{eqnarray*}
Hence we have proved that $
[\,\Delta(x_0^+(-1)),\,\Delta(x_0^-(1))\,]
%&=&[\,x_0^+(-1),\,x_0^-(1)\,]\ot \gamma^{-1}K_0+\gamma K_0^{-1}\ot[\,x_0^+(-1),\,x_0^-(1)\,]\\
=\frac{\Delta(\gamma^{-1} K_0)-\Delta(\gamma K_0^{-1})}{q-q^{-1}}$.
\end{proof}
For the quantum toroidal algebra $U_q(sl_2, tor)$, we define that
\begin{eqnarray*}
&&S(x_0^+(-1))=-\gamma^{-1} K_0x_0^+(-1)+(q-q^{-1})q^2\gamma^{-1}K_0K_1x_1^{-}(0)[x_0^+(-1),\,x_1^+(0)]_q,\\
&&S(x_0^-(1))=-x_0^-(1)\gamma K_0^{-1}-(q-q^{-1})q^{-2}[x_{1}^-(0),\,x_0^-(1)]_{q^{-1}}x_1^+(0)\gamma K_0^{-1}K_1^{-1}.
\end{eqnarray*}

For the quantum toroidal algebra $U_q(sl_3, tor)$, we get that
\begin{eqnarray*}
&&S(x_0^+(-1))=-\gamma^{-1} K_0x_0^+(-1)+(q-q^{-1})q^2\gamma^{-1}K_0K_1x_1^{-}(0)[x_0^+(-1),\,x_1^+(0)]_q,\\
&&\hskip3.1cm-(q-q^{-1})q^3\gamma^{-1}K_0K_1K_2[x_1^{-}(0),x_2^-(0)]_{q^{-1}}[[x_0^+(-1),\,x_1^+(0)]_q,\,x_2^+(0)]_{q},\\
&&S(x_0^-(1))=-x_0^-(1)\gamma K_0^{-1}-(q-q^{-1})q^{-2}[x_{1}^-(0),\,x_0^-(1)]_{q^{-1}}x_1^+(0)\gamma K_0^{-1}K_1^{-1}\\
&&\hskip1.8cm-(q-q^{-1})q^{-3}[x_2^-(0),\,[x_{1}^-(0),\,x_0^-(1)]_{q^{-1}}]_{q^{-1}}[x_2^+(0),\,x_1^+(0)]_{q}\gamma K_0^{-1}K_1^{-1}K_2^{-1}.
\end{eqnarray*}

The following result is directly computed by the comultiplication and definition.

\begin{prop}
Under the action of the antipode $s$ on $x_0^+(-1)$ and $x_0^-(1)$, the algebra ${\mathcal U}_0$  for $n=1$ or $n=2$  is a Hopf algebra with
the comultiplication $\Delta$ given above, the counit $\vep$ and the antipode $S$
defined below. For $i\in I$.
\begin{eqnarray*}
&&\varepsilon(x_i^+(0))=\varepsilon(x_i^-(0))=\varepsilon(x_0^+(-1))=\varepsilon(x_0^-(1))=0,\\
&& \varepsilon(\gamma^{\pm\frac{1}2})
=\varepsilon(K_i)=\varepsilon(K_i^{-1})=\varepsilon(q^{d_1})=\varepsilon(q^{d_2})=1,\\
&&S(\gamma^{\pm\frac{1}2})=\gamma^{\mp\frac{1}2},\qquad S(K_i^{\pm1})=K_i^{\mp1}, \qquad S(q^{d_1})=q^{-d_1},\qquad S(q^{d_2})=q^{-d_2},\\
&&S(x_i^+(0))=-K_i^{-1}x_i^+(0),\qquad S(x_i^-(0))=-x_i^-(0)\,K_i.
\end{eqnarray*}
\end{prop}

For general case,
We begin with the notations. Denote $\Pi_0$ by the set of all real root vectors of toroidal Lie algebra
$(sl_{n+1},tor)$ generated by simple roots beginning with $\alpha_0$, or all real roots supported at
$\alpha_0$ with multiplicity one. Note that $\Pi_0$ is a finite set, in fact, if one bounds the coefficients
of any fixed simple root, then there are only finitely many such affine roots.
For any real root $\alpha_{\beta 0}\in \Pi_0$
suppose $\alpha_{\beta 0}$ obtained from the simple roots $\alpha_0, \alpha_{i_1}, \cdots, \alpha_{i_k}$, that is, $\alpha_{\beta 0}=\alpha_0+\alpha_{i_1}+\cdots+\alpha_{i_k}$, where $i_1,\cdots, i_k\in \{1, \cdots, n\}$.

\vskip 0.2in

%TeXCAD Options
%\grade{\on}
%\emlines{\off}
%\epic{\off}
%\beziermacro{\on}
%\reduce{\on}
%\snapping{\off}
%\quality{8.00}
%\graddiff{0.01}
%\snapasp{1}
%\zoom{4.0000}
\unitlength 1mm % = 2.85pt
\linethickness{0.4pt}
\ifx\plotpoint\undefined\newsavebox{\plotpoint}\fi % GNUPLOT compatibility
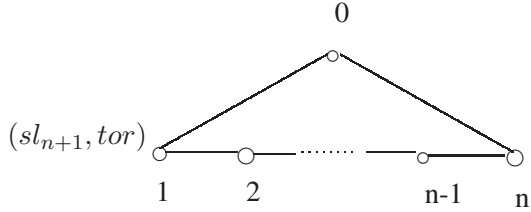
\begin{picture}(74.5,27.75)(0,0)
\put(6.0,11.5){$(sl_{n+1}, tor)$}
\put(26.25,10.25){\circle{1.8}}
\put(37.75,10){\circle{2.06}}
\put(61.25,9.75){\circle{1.5}}
\put(73.5,9.75){\circle{2}}
\put(49.25,23.25){\circle{1.5}}
%\emline(27,10.5)(36.5,10.5)
\put(27,10.5){\line(1,0){9.5}}
%\end
%\emline(38.75,10.25)(44.25,10.25)
\put(38.75,10.25){\line(1,0){5.5}}
%\end
%\emline(53.75,10.5)(60.25,10.5)
\put(53.75,10.5){\line(1,0){6.5}}
%\end
%\emline(62,10)(72,10)
\put(62,10){\line(1,0){10}}
%\end
%\dottedline(45.25,10.5)(51.75,10.5)
\multiput(45.18,10.43)(.9286,0){8}{{\rule{.4pt}{.4pt}}}
%\end
%\emline(26.25,11)(48.5,23.5)
\multiput(26.25,11)(.059973046,.033692722){371}{\line(1,0){.059973046}}
%\end
%\emline(73.25,10.5)(50.25,23.5)
\multiput(73.25,10.5)(-.059585492,.033678756){386}{\line(-1,0){.059585492}}
%\end
\put(25.75,4){1}
\put(37.75,4){2}
\put(61.5,3.75){n-1}
\put(73.5,3.25){n}
\put(49.5,27.75){0}
\end{picture}

We remark that the roots in $\Pi_0$ are in fact real roots of some affine Lie algebra in type $A$.
In fact, the real roots $\alpha_{\beta 0}\in \Pi_0$ can always be written in three possible ways:
 $$\alpha_{\beta 0}=\alpha_0+\alpha_1+\alpha_2\cdots+\alpha_{k_1}\in\Pi_1,$$
$$\alpha_{\beta 0}=\alpha_0+\alpha_n+\alpha_{n-1}\cdots+\alpha_{k_2}\in\Pi_2,$$
$$\alpha_{\beta 0}=\alpha_0+\alpha_1+\alpha_n+\cdots+\alpha_{k_3}\in\Pi_3,$$  where $k_i=2,\cdots, n-1$ for $i=1, 2, 3$.
They respectively correspond to three cases of positive roots passing through a fixed simple root. For example, in type $A_5$
any root passing through $\alpha_2$ are: $\alpha_2+\alpha_3+\cdots$, $\alpha_1+\alpha_2$, or $\alpha_1+\alpha_2+\alpha_3+\cdots$.

Therefore, $\Pi_0=\Pi_1\cup\Pi_2\cup\Pi_3$. Consider the real root system $\Pi_0$ by induction, except for the special
cases of $n=1$ and $n=2$, the real root system $\Pi_0$ of the cases of $n>2$ would degenerate to the case for $n=3$, because
$\Pi_i$ can be obtained by the same way for $i=1, 2, 3$. Therefore, the proof of Proposition \ref{p:5.4} below will be divided into three
cases $n=1$, $n=2$ and $n=3$.

Accordingly we define the quantum root vectors $x_{\alpha_{\beta 0}}^-(1)$, $x_{\alpha_{\beta 0}}^+(-1)$, $y_{\alpha_{\beta 0}}^-(1)$ and $y_{\alpha_{\beta 0}}^+(-1)$ associated to the root $\alpha_{\beta 0}$ as follows:
\begin{align*}
x_{\alpha_{\beta 0}}^-(1)&=[\,x_{i_{k}}^-(0),\, x_{i_{k-1}}^-(0),\,\cdots,\, x_{i_1}^-(0),\,x_0^-(1)\,]_{(q^{u_{i_1}},\,\cdots,\, q^{u_{i_k}})},\\
x_{\alpha_{\beta 0}}^+(-1)&=[\,x_0^+(-1),\,x_{i_1}^+(0),\, x_{i_2}^+(0),\,\cdots,\, x_{i_{k-1}}^+(0),\,x_{i_{k}}^+(0)\,]_{(q^{v_{i_{k-1}}},\,\cdots, q^{v_{i_1}},\, q^{u_{i_k}})},\\
y_{\alpha_{\beta 0}}^-(1)&=[\,x_{i_{k}}^-(0),\, x_{i_{k-1}}^-(0),\,\cdots,\, x_{i_1}^-(0),\,x_0^-(1)\,]_{\la q^{-v_{i_{k-1}}},\,\cdots,\, q^{-v_{i_1}},\,q^{-u_{i_k}}\ra},\\
y_{\alpha_{\beta 0}}^+(-1)&=[\,x_0^+(-1),\,x_{i_1}^+(0),\, x_{i_2}^+(0),\,\cdots,\, x_{i_{k-1}}^+(0),\,x_{i_{k}}^+(0)\,]_{\la q^{-u_{i_{1}}},\,\cdots,\, q^{-u_{i_k}}\ra},
\end{align*}
where for $j=1, \cdots, k$
\begin{align*}
u_{i_j}&=((\alpha_0,\,\alpha_{i_1})+(\alpha_0,\,\alpha_{i_2})\cdots+(\alpha_0,\,\alpha_{i_j})),\\ v_{i_j}&=-((\alpha_{i_k},\,\alpha_{i_{k-1}})+(\alpha_{i_k},\,\alpha_{i_{k-2}})\cdots+(\alpha_{i_k},\,\alpha_{i_{j}})).
\end{align*}
Define  $\xi_{\alpha_{\beta 0}}=\frac{1}{[-u_{i_k}]}$.

We define the diagram automorphism $\sigma$ of the algebra ${\mathcal U}_0$ as follows: $\sigma(i)=n+1-i$ for $1\leqslant i \leqslant n$. Note that  for any $\alpha_{\beta 0}\in \Pi_0$, and $\sigma(\alpha_{\beta 0})\in \Pi_0$.

\medskip

The following are easy consequences of \eqref{e:Dr8} %e:comm4}
and \eqref{e:comm0}.
\begin{lemm} With the above notations, one gets that
\begin{align}
&[x_0^-(1), x_{\alpha_{\beta 0}}^+(-1)]=[-u_{i_k}][\,x_{i_{1}}^+(0),\,\cdots, x_{i_{k-1}}^+(0),\, x_{i_k}^+(0)\,]_{(q^{v_{i_{k-1}}},\,\cdots,\, q^{v_{i_1}})}\gamma^{-1} K_0,\\
&[y_{\alpha_{\beta 0}}^-(1),\,x_0^+(-1)]=[u_{i_k}]\gamma K_0^{-1}[\,x_{i_{k}}^-(0),\,\cdots, x_{i_{2}}^-(0),\, x_{i_1}^-(0)\,]_{\la q^{-v_{i_{k-1}}},\,\cdots,\, q^{-v_{i_1}}\ra},\\
&[x_0^+(-1), x_{\alpha_{\beta 0}}^-(1)]
=q^{u_{i_k}-u_{i_1}}[-u_{i_1}]\gamma^{-1}K_0[\,x_{i_{k}}^-(0), \cdots, x_{i_1}^-(0)\,]_{(q^{u_{i_1}}, \cdots, q^{u_{i_{k-1}}})}, \\
&[y_{\alpha_{\beta 0}}^+(-1),\,x_0^-(1)]=q^{-u_{i_k}+u_{i_1}}[u_{i_1}][\,x_{i_{1}}^+(0),\,\cdots, x_{i_k}^+(0)\,]_{\la q^{-u_{i_1}},\,\cdots,\, q^{-u_{i_{k-1}}}\ra}\gamma^{-1} K_0.
\end{align}
\end{lemm}
%\begin{proof} The lemma is verified directly by using relations \eqref{e:comm4} and \eqref{e:comm0a}. %$(3.3)$.
%\end{proof}

We
now define a comultiplication map ${\mathcal U}_0\longrightarrow {\mathcal U}_0\otimes
{\mathcal U}_0$ on the generators as follows.
\begin{align}\label{e:commult1}
&\Delta(K_i)=K_i\ot K_i, \qquad \Delta(\gamma^{\pm\frac{1}2})=\gamma^{\pm\frac{1}2}\otimes
\gamma^{\pm\frac{1}2},\qquad
\Delta(q^{d_1})=q^{d_1}\ot q^{d_1}, \\ \label{e:commult2}
%\Delta(q^{d_2})=q^{d_2}\ot q^{d_2}, \\
&\Delta(x_i^+(0))=x_i^+(0)\ot 1+K_i\ot x_i^+(0), \qquad
\Delta(x_i^-(0))=x_i^-(0)\ot K_i^{-1}+1\ot
x_i^-(0), \\ \notag
&\Delta(x_0^-(1))=x_0^-(1)\ot\gamma^{-1}K_0+1\ot x_0^-(1)\\ \label{e:commult3}
&\hskip1.9cm -(q-q^{-1})\sum_{\alpha_{\beta 0}\in \Pi_0}\xi_{\alpha_{\beta 0}}x_{\alpha_{\beta 0}}^-(1)\ot [x_0^-(1),\, x_{\alpha_{\beta 0}}^+(-1)],\\ \notag
&\Delta(x_0^+(-1))=x_0^+(-1)\ot 1+\gamma K_0^{-1}\ot x_0^+(-1)\\ \label{e:commult4}
&\hskip1.9cm-(q-q^{-1})\sum_{\alpha_{\beta 0}\in \Pi_0}\xi_{\alpha_{\beta 0}}[y_{\alpha_{\beta 0}}^-(1),\,x_0^+(-1)]\ot y_{\alpha_{\beta 0}}^+(-1).
\end{align}
\begin{remark}\, Note that $\Pi_0$ is finite. The coproduct formulas can also be formulated as a sum over different paths
along the Dynkin diagram of type
$A_n^{(1)}$.
\end{remark}

%Now we have the following conjecture.

\noindent{\bf Conjuecture:}
The algebra ${\mathcal U}_0$  for $n\geq 1$
is a Hopf algebra
under the comultiplication formulas.

The conjecture holds when $n=1, 2$ due to Prop. \ref{p:5.4}.

\section{Quantum double algebra structure of ${\mathcal U}_0$}
We give the Drinfeld double structure for the algebra ${\mathcal U}_0$ in this section.

\begin{defi}
A bilinear form $\langle\ \, ,\ \rangle:$ ${\mathfrak{B}}\times
{\mathfrak{A}}\longrightarrow \mathbb{K}$ is called a skew-dual pairing
of two Hopf algebras ${\mathfrak{A}}$ and ${\mathfrak{B}}$, %(see \cite{KS}),
if it satisfies
\begin{gather*}
\langle b,\,
1_{{\mathfrak{A}}}\rangle=\varepsilon_{{\mathfrak{B}}}(b),\qquad\qquad
\langle 1_{{\mathfrak{B}}},\, a\rangle=\varepsilon_{{\mathfrak{A}}}(a),\\
\langle b,\, a_1a_2\rangle=\langle\Delta^{{\rm op}}_{{\mathfrak{B}}}(b),
a_1\otimes a_2\rangle,\qquad \langle b_1b_2, a\rangle=\langle
b_1\otimes b_2, \Delta_{{\mathfrak{A}}}(a)\rangle,
\end{gather*}
for all $a,\, a_1,\, a_2\in{{\mathfrak{A}}}$ and $b,\, b_1,\,
b_2\in{{\mathfrak{B}}}$, where $\varepsilon_{{\mathfrak{A}}}$,
$\varepsilon_{{\mathfrak{B}}}$ denote the counites of ${\mathfrak{A}}$,
${\mathfrak{B}}$, respectively, and $\Delta_{{\mathfrak{A}}}$,
$\Delta_{{\mathfrak{B}}}$ are the respective comultiplications.
\end{defi}

In \cite{G}, Gross\'e defined a weak Hopf pairing for the quantum affine algebras within
a much larger topological Hopf algebra. It was not clear if
the formulas there are closed on the quantum affine algebras. Our Hopf pairing is for the quantum toroidal algebra and closed on the
subalgebra of the quantum affine algebra.

A direct consequence of the definition is that
$$\langle S_{\mathfrak{B}}(f),a \rangle=
 \langle f,S_{\mathfrak{A}}(a) \rangle, \quad f \in {\mathfrak{B}},\quad a \in {\mathfrak{A}},$$
  where $S_{\mathfrak{B}}, S_{\mathfrak{A}}$ denote the
 antipodes of $\mathfrak{B}$ and $\mathfrak{A}$, respectively.

\begin{defi} \label{b:1} For any two skew-paired Hopf algebras ${\mathfrak{A}}$ and
${\mathfrak{B}}$ (see \cite{KS}) by $\langle\, , \rangle$, there exits a
Drinfeld quantum double ${\mathcal D}({\mathfrak{A}}, {\mathfrak{B}})$ which is
a Hopf algebra whose underlying coalgebra is
${\mathfrak{A}}\otimes{\mathfrak{B}}$ with the tensor product coalgebra
structure, and its algebra structure is defined by
$$
(a\otimes b)(a'\otimes b')=\sum\langle S_{\mathfrak{B}}(b_{(1)}),
a'_{(1)} \rangle\langle b_{(3)}, a'_{(3)}\rangle aa'_{(2)}\otimes
b_{(2)}b', $$
 for $a, a' \in {\mathfrak{A}}$ and $b, b' \in {\mathfrak{B}}$, and whose
 antipode $S$ is given by
$$
S(a\otimes b)=(1\otimes S_{\mathfrak{B}}(b))(S_{\mathfrak{A}}(a) \otimes 1).
$$
\end{defi}

%\begin{defi} A bilinear form $\langle \, , \,
%\rangle : \mathcal{U}\times
% \mathcal{A}\longrightarrow \mathbb{K}$ is called a skew-dual pairing of
%two Hopf algebras $\mathcal{U}$ and $\mathcal{A}$,
%if $\langle , \rangle:
% {\mathcal{U}}^{cop}\times\mathcal{A}\longrightarrow \mathbb{C}(q^{\pm1})$ is a Hopf dual pairing of
% $\mathcal{A}$ and ${\mathcal{U}}^{cop}$, where
%${\mathcal{U}}^{cop}$ is the Hopf algebra having the opposite
%comultiplication to $\mathcal{U}$, and if $S_{\mathcal{U}}$ is
%invertible, then $S_{{\mathcal{U}}^{cop}}=S_{\mathcal{U}}^{-1}$.
% \end{defi}

Let $\mathcal B$ (resp. $\mathcal B'$) denote the Hopf
(Borel-type) algebra of $\mathcal{U}_0$ for $n=1$ or $n=2$ generated by
$x_j^+(0)$, $x_0^-(1),\, K_j^{\pm 1}$, $q^{d_1},\, q^{d_2},\,\gamma^{\pm\frac{1}2}$ (resp.
$x_j^-(0),\,x_0^+(-1),\, K_j^{\pm1}, \, q^{d_1},\, q^{d_2},\, \gamma^{\pm\frac{1}2}$) with $j\in I$.

\begin{prop}{\label{b:3} There exists a unique skew-dual pairing
$\langle \,, \rangle:\mathcal B' \times \mathcal
B\longrightarrow\mathbb{C}(q^{\pm\frac{1}{2}})$ of the Hopf subalgebras $\mathcal
B$ and $\mathcal B'$ such that $(i\in I)$:
\begin{gather}\label{e:pair1}
\langle x_i^-(0),\,x_j^+(0)\rangle=\delta_{i j}\frac{1}{q^{-1}-q}, \\ \label{e:pair2}
\langle x_0^-(1),\,x_0^+(-1)\rangle=\frac{1}{q^{-1}-q}, \\ \label{e:pair3}
\langle K_i^{-1},\,K_j\rangle^{\pm 1}=\langle K_i^{-1},\,K_j^{-1}\rangle^{\mp 1}=\langle K_i,\,K_j^{-1}\rangle^{\pm 1}=\langle K_i,\,K_j\rangle^{\mp 1}=q^{\pm a_{ij}},\\ \label{e:pair4}
\langle\gamma^{\pm\frac{1}{2}},\, \gamma^{\pm\frac{1}{2}}\rangle=1=\langle\gamma^{\pm\frac{1}{2}},\, \gamma^{\mp\frac{1}{2}}\rangle,\\ \label{e:pair5}
\langle q^{d_1},\, q^{d_1}\rangle=\langle q^{d_1},\, q^{d_2}\rangle=1=\langle q^{d_2},\, q^{d_1}\rangle=\langle q^{d_2},\, q^{d_2}\rangle,\\ \label{e:pair6}
\langle \gamma^{\pm\frac{1}2},\, K_i^{\pm}\rangle=\langle K_i^{\pm},\,\gamma^{\pm\frac{1}2}\rangle=1=\langle
K_i^{\pm1},\, \gamma^{\mp\frac{1}2}\rangle=\langle\gamma^{\mp\frac{1}2},\,K_i^{\pm1}\rangle,\\ \label{e:pair7}
\langle q^{d_1},\,K_i^{\pm1}\rangle=\langle q^{d_2},\,K_i^{\pm1}\rangle=q^{\delta_{0i}},\\ \label{e:pair8}
\langle K_i^{\pm1},\,q^{d_1}\rangle=\langle K_i^{\pm1},\,q^{d_2}\rangle=q^{-\delta_{0i}},\\ \label{e:pair9}
\langle q^{d_1},\,\gamma^{\pm\frac{1}2}\rangle=\langle q^{d_2},\,\gamma^{\pm\frac{1}2}\rangle=q^{\pm\frac{1}{2}}=
\langle \gamma^{\pm\frac{1}2},\,  q^{d_1}\rangle=\langle \gamma^{\pm\frac{1}2},\,  q^{d_2}\rangle,
\end{gather}
 and all others pairs of generators are $0$. Moreover,
$\langle S(b'), S(b)\rangle=\langle b', b\rangle$ for $b'\in
\mathcal B', ~b\in\mathcal B$. }
\end{prop}
%{\color{red} please check the exponent in (6.3)}

\begin{proof} \ The uniqueness is clear, as any
skew-dual pairing of a bialgebra is determined by the value on the
generators. We proceed to prove the existence of the Hopf pairing.

The pairing defined on the generators as given in \eqref{e:pair1}-\eqref{e:pair9} can be
extended to a bilinear form on $\mathcal B' \times
\mathcal B$ such that the defining properties in
Definition \ref{b:1} hold. We will verify that the relations in
$\mathcal B$ and $\mathcal B'$ are preserved, ensuring
that the form is well-defined and is a skew-dual pairing of
$\mathcal B$ and $\mathcal B'$.

First of all,  it is straightforward to check that the bilinear form
preserves the relation \eqref{e:comm0} for $K_i^{\pm1}$,
$\gamma^{\pm\frac{1}2}\in\mathcal B$ and
${K_i}^{\pm1}$, $\gamma^{\,\pm\frac{1}2}\in\mathcal B'$. Similar it also preserves \eqref{e:comm00} in
$\mathcal B$ or $\mathcal B'$ respectively.
So we are
left to verify that the bilinear form preserves the quantum Serre relations
in $\mathcal B$ and $\mathcal B'$.

Using \cite{HRZ}, we only need to verify the relations involving generators $x_i^{\pm1}(\mp1)$ in $\mathcal
B$ and $\mathcal B'$. Namely we need to check that
$$\langle X, (x_1^+(0))^2x_0^+(-1)-[2]x_1^+(0)x_0^+(-1)x_1^+(0)+x_0^+(-1)(x_1^+(0))^2\rangle=0,$$
for any word $X$ in the generators of  $\mathcal B'$. By
definition, this equals to
\begin{equation}{\label{b:2}
\begin{split}
\langle \Delta^{(2)}(X),\,& \, x_1^+(0)\otimes x_1^+(0)\otimes
x_0^+(-1)\\
&-[2]x_1^+(0)\otimes x_0^+(-1)\otimes x_1^+(0)+x_0^+(-1)\otimes x_1^+(0) \otimes x_1^+(0)\rangle,
\end{split}}
\end{equation}
where the $\Delta$ corresponds to $\Delta_{\mathcal B'}^{{\rm op}}$. If
one of these terms were nonzero, $X$ would have involved
exactly two $x_1^-(0)$ factors, one $x_0^+(1)$ factor, and arbitrarily
many $K_j^{\pm1}$ $(j\in I)$ or $\gamma^{\,\pm\frac{1}2}$ factors.
For simplicity, we first consider three
cases:

(i) \ If $X=x_1^-(0)^2 x_0^-(1)$, then $\Delta^{(2)}(X)$ is equal to
\begin{eqnarray*}
&(K_0^{-1}\otimes K_0^{-1} \otimes x_1^-(0)+K_0^{-1}\otimes x_1^-(0)
\otimes 1+ x_1^-(0)\otimes 1 \otimes 1)^2(K_0\otimes
K_0\otimes x_0^-(1)\\
&\quad+K_0 \otimes  x_0^-(1)\otimes 1+ x_0^-(1)\otimes 1\otimes
1 -\sum_{\ell=1}^{n-1}(q-q^{-1})x_{1,\,\ell}^+(0)K_0\ot
x_{\ell, 0}^-(1)\ot 1\\
&\quad-\sum_{\ell=1}^{n-1}(q-q^{-1})\Delta(x_{1,\,\ell}^+(0)K_0)\ot
x_{\ell, 0}^-(1)).
\end{eqnarray*}
 The relevant terms of
$\Delta^{(2)}(X)$ are
\begin{gather*}
x_1^-(0)K_1^{-1}K_0\otimes x_1^-(0)K_0\otimes x_0^-(1)+
K_1^{-1}x_1^-(0)K_0 \otimes
x_1^-(0)K_0\otimes x_0^-(1)\\
+ x_1^-(0)K_1^{-1}K_0\otimes K_1^{-1}x_0^-(1)\otimes x_1^-(0)+
K_1^{-1} x_1^-(0)K_0\otimes K_1^{-1} x_0^-(1)\otimes x_1^-(0)\\
+ K_1^{-1}K_1^{-1} x_0^-(1) \otimes x_1^-(0)K_1^{-1}\otimes  x_1^-(0)+
K_1^{-1}K_1^{-1} x_0^-(1)\otimes K_1^{-1} x_1^-(0)\otimes  x_1^-(0).
\end{gather*}
Therefore $(\ref{b:2})$ becomes
\begin{equation*}
\begin{split}
\langle x_1^-(0)\,&K_1^{-1}K_0,\, x_1^+(0) \rangle\langle
x_1^-(0)K_0,\, x_1^+(0) \rangle\langle x_0^-(1),\, x_0^+(-1)
\rangle\\
&\quad+\langle K_1^{-1}x_1^-(0)K_0,\, x_1^+(0)\rangle\langle
 x_1^-(0)K_{0},\, x_1^+(0)\rangle \langle x_0^-(1),\, x_0^+(-1)\rangle \\
&\quad-[2]\bigl(\langle x_1^-(0)K_1^{-1}K_0,\,x_1^+(0)\rangle
 \langle K_1^{-1}x_0^-(1),\, x_0^+(-1)\rangle\langle
x_1^-(0),\,x_1^+(0) \rangle\\
&\quad+\langle K_1^{-1}x_1^-(0)K_0,\, x_1^+(0)\rangle\langle
K_1^{-1}x_0^-(1),\, x_0^+(-1) \rangle\langle x_1^-(0),
x_1^+(0)\rangle\bigr)\\
&\quad + \bigl(\langle {K}_1^{-1}{K}_1^{-1}x_0^-(1),\, x_0^+(-1)
\rangle\langle x_1^-(0)K_1^{-1},\,x_1^+(0)\rangle\langle x_1^-(0),\,x_1^+(0)
\rangle\\
&\quad+\langle
 K_1^{-1}K_1^{-1}x_0^-(1),\,x_0^+(-1)\rangle\langle K_1^{-1}x_1^-(0),\,x_1^+(0) \rangle\langle x_1^-(0),\,x_1^+(0)\rangle\bigr)\\
%&=\frac{1}{(q^{-1}-q)^3}\bigl\{1+\langle K_1^{-1},\,K_1\rangle-(q{+}q^{-1})\bigl(
% \langle K_1^{-1},\, K_0^{-1}\rangle+
% \langle K_1^{-1},\, K_1 \rangle \\
%&\quad\langle K_1^{-1},\, K_0^{-1}\rangle\bigr)+\bigl(\langle K_1^{-1},\, K_0^{-1}\rangle^2+
% \langle K_1^{-1},\, K_0^{-1}\rangle^2\langle K_1^{-1},
% K_1\rangle\bigr)\bigr\}\\
&=\frac{1}{(q^{-1}-q)^3}\bigl\{1+q^{2}-(q{+}q^{-1})(q+q^{3})+(q^2+
 q^4)\bigr\}=0.
\end{split}
\end{equation*}

(ii) \ When $X=x_1^-(0)x_0^-(1)x_1^-(0)$, it is easy to get the relevant terms
of $\Delta^{(2)}(X)$:
\begin{gather*}
K_1^{-1}K_0x_1^-(0)\otimes x_1^-(0)K_0\otimes x_0^-(1)+
x_1^-(0)K_0K_1^{-1}\otimes K_0x_1^-(0)\otimes x_0^-(1)\\
+K_1^{-1}K_0x_1^-(0)\otimes K_1^{-1}x_0^-(1)\otimes x_1^-(0) +x_1^-(0)
K_0K_1^{-1}\otimes x_0^-(1)K_1^{-1} \otimes x_1^-(0)\\
+K_1^{-1}x_0^-(1)K_1^{-1} \otimes K_1^{-1}x_1^-(0)\otimes x_1^-(0)+
K_1^{-1}x_0^-(1)K_1^{-1}\otimes x_1^-(0)K_1^{-1} \otimes x_1^-(0).
\end{gather*}
Thus, \eqref{b:2} becomes
\begin{equation*}
\begin{split}
&\frac{1}{(q^{-1}{-}q)^3}\bigl\{\langle K_1^{-1},\,
K_1\rangle\langle K_0,\, K_1\rangle
 +\langle K_0,\, K_1\rangle-(q{+}q^{-1}) (\langle K_1^{-1},\, K_1\rangle
 \langle K_0,\, K_1\rangle\langle  K_1^{-1},\, K_0\rangle +1)\\
&\qquad\qquad +\bigl(\langle K_1^{-1},\,
K_0 \rangle\langle  K_1^{-1},\,K_1\rangle+
 \langle K_1^{-1},\, K_0\rangle\bigr)\bigr\}\\
 &=\frac{1}{(q^{-1}{-}q)^3}\bigl\{(q^3+q)-(q+q^{-1})(q^2+1)+(q+q^{-1})\bigr\}=0.
\end{split}
\end{equation*}

(iii) \ If $X=x_0^-(1)x_1^-(0)x_1^-(0)$, one can similarly get that
$(\ref{b:2})$ vanishes.

\medskip
Finally, if $X$ is any word involving exactly two $x_1^-(0)$ factors, one
$x_0^-(1)$ factor, and arbitrarily many factors $K_j^{\, \pm1}$
$(j\in I$)  and $\gamma^{\,\pm\frac{1}2}$, then
$(2.2)$ will just be a scalar multiple of one of the quantities we
have already calculated, and then it will be zero.

Analogous calculations show that the relations in $\mathcal
B'$ are preserved.
\end{proof}

The following can be proved similarly as \cite{BGH}.
\begin{theo} {\label{b:4}
${\mathcal D}({\mathcal B}, {\mathcal B'})$ is isomorphic to
${\mathcal U}_0$  for $n=1$ or $n=2$ as a Hopf algebra.}
\end{theo}
%\begin{proof} We will denote the image $x_i^+(0)\otimes 1$ of
%$x_i^+(0)$ in ${\mathcal D}({\mathcal B}, {\mathcal B'})$ by
%${\widehat{x_i^+(0)}}$,
%and denote the image $x_0^+(-1)\otimes 1$ of
%$x_0^+(-1)$ in ${\mathcal D}({\mathcal B}, {\mathcal B'})$ by
%${\widehat{x_0^+(-1)}}$
%and similarly for $K_i^{\pm1},
%\gamma^{\pm\frac{1}2}$, denote the image $1\otimes x_i^-(0)$ of
%$x_i^-(0)$ in ${\mathcal D}({\mathcal B}, {\mathcal B'})$ by
%${\widehat{x_i^-(0)}}$, and denote the image $x_0^-(1)\otimes 1$ of
%$x_0^-(1)$ in ${\mathcal D}({\mathcal B}, {\mathcal B'})$ by
%${\widehat{x_0^-(1)}}$
% and similarly for $K_i^{\,\pm1},
%\gamma^{\,\pm\frac{1}2}$. Define a map $\varphi: {\mathcal
%D}({\mathcal B}, {\mathcal B'})\longrightarrow
%\mathcal{U}_0$ by
%\begin{gather*}
%\varphi({\widehat{x_i^+(0)}})=x_i^+(0),\quad\varphi({\widehat{x_i^-(0)}})=x_i^+(0),\quad
%\varphi(\widehat{x_0^+(-1)})=x_0^+(-1),\\
%\varphi(\widehat{x_0^-(1)})=x_0^-(1)\quad,\varphi(\widehat{K}_i^{\pm1})=K_i^{\pm1},\quad
%\varphi({\widehat{\gamma}^{\pm\frac{1}{2}}})=\gamma^{\pm\frac{1}{2}}.
%\end{gather*}
%
%\end{proof}

\begin{remark} Using the Hopf double structure, one can show the existence of
a universal R-matrix for the Hopf algebra ${\mathcal U}_0$.
The universal $R$-matrix \cite{D1} for the Hopf algebra $U$ is
an invertible element $R$ of $U\widehat{\otimes} U$ satisfying the conditions
\begin{align*}
\Delta(x) &= R\Delta(x)R^{-1}, \quad \forall x\in U,\\
(\Delta\otimes id)R &= R_{13}R_{23},\quad  (id\otimes\Delta)R = R_{13}R_{12},
\end{align*}
where $\Delta= \Delta^{op}$ is the opposite comultiplication in $U$. In this case $U$ is a
quantum double of a Hopf algebra $U^+, U\simeq U^+\otimes U^-, U^-$ being dual to $U^+$ with an
opposite comultiplication, then $U$ admits a canonical presentation of the universal R-matrix
$R=\sum e_i\otimes f_i$, where $e_i$ and $f_i$ are dual bases of $U^+$ and $U^-$.
\end{remark}
\smallskip

\vskip30pt \centerline{\bf ACKNOWLEDGMENT}

N. Jing would like to thank the partial support of
Simons Foundation grant 523868, NSFC grant 11531004 and 12171303, and NSF grants
1014554 and 1137837. H. Zhang would
like to thank the support of NSFC grant 11871325.
\bigskip

\bibliographystyle{amsalpha}

\end{document}